\documentclass[11pt]{amsart}
\usepackage{amsthm,amsmath,stmaryrd,bbm,geometry,color}
\usepackage{amssymb}
\usepackage[english]{babel}
 \usepackage[utf8]{inputenc}
\usepackage{graphicx}
\usepackage{verbatim}
\usepackage{enumitem}
\usepackage{mathtools}
\usepackage[titletoc]{appendix}
\usepackage{bm}
\usepackage{textcomp}
\usepackage{amsaddr}

\usepackage[dvipsnames]{xcolor}
\usepackage[hyperindex=true,frenchlinks=true,colorlinks=true,
citecolor=Mahogany,linkcolor=Red,urlcolor=Tan,linktocpage]{hyperref}

\usepackage{tikz}
\usetikzlibrary{shapes}
\usetikzlibrary{fit}
\usetikzlibrary{decorations.pathmorphing}

\allowdisplaybreaks

\title{Deviation and moment inequalities for Banach-valued
$U$-statistics}
\author{Davide Giraudo}
\address{Institut de Recherche Mathématique Avancée UMR 7501, Université de
Strasbourg and CNRS 7 rue René Descartes 67000 Strasbourg, France }
\email{ dgiraudo@unistra.fr}
\date{\today}

\numberwithin{equation}{section}
\renewcommand{\leq}{\leqslant}
\renewcommand{\geq}{\geqslant}

\newtheorem{Theorem}{Theorem}[section]
\newtheorem{Proposition}[Theorem]{Proposition}
\newtheorem{Lemma}[Theorem]{Lemma}
\newtheorem{Definition}[Theorem]{Definition}
\newtheorem{Corollary}[Theorem]{Corollary}

\theoremstyle{remark}

\tikzstyle{Vertex}=[circle,draw=LimeGreen!80,fill=LimeGreen!8,
inner sep=1pt,minimum size=2mm,line width=1pt,font=\scriptsize]
\tikzstyle{Node}=[Vertex,draw=RoyalBlue!80,fill=RoyalBlue!8,inner sep=1.5pt]
\tikzstyle{Leaf}=[rectangle,draw=Black!70,fill=Black!16,
inner sep=0pt,minimum size=1mm,line width=1.25pt]
\tikzstyle{Edge}=[Maroon!80,cap=round,line width=1pt]
\tikzstyle{Mark1}=[draw=BrickRed!80,fill=BrickRed!8]
\tikzstyle{Mark2}=[draw=BurntOrange!80,fill=BurntOrange!8]
\tikzstyle{EdgeRew}=[->,RedOrange!80,cap=round,thick]

\newcommand{\Aca}{\mathcal{A}}
\newcommand{\Bca}{\mathcal{B}}

\newcommand{\Fca}{\mathcal{F}}
\newcommand{\Gca}{\mathcal{G}}
\newcommand{\Hca}{\mathcal{H}}

\newcommand{\Sca}{\mathcal{S}}
\newcommand{\Uca}{\mathcal{U}}

\newcommand{\intent}[1]{\llbracket #1\rrbracket}
\newcommand{\B}{\mathbb{B}}
\newcommand \ens[1]{\left\{ #1\right\}}

\newcommand \R{\mathbb R}

\newcommand \N{\mathbb N}
\newcommand \PP{\mathbb P}
\newcommand{\el}{\mathbb L}
\newcommand{\E}[1]{\mathbb E\left[#1\right]}

\newcommand \Z{\mathbb Z}

\newcommand \abs[1]{\left|#1\right|}
\newcommand \eps{\varepsilon}
\newcommand{\lb}{\lambda}

\newcommand{\pr}[1]{\left(#1\right)}
\newcommand{\norm}[1]{\left\lVert #1 \right\rVert}
\newcommand{\gr}[1]{\bm{#1}}

\newcommand{\gri}{\gr{i}}
\newcommand{\grj}{\gr{j}}

\newcommand{\inc}{\operatorname{Inc}}
\newcommand{\ind}[1]{\mathbf{1}_{#1}}
\newcommand{\til}[1]{\widetilde{#1}}

\newcommand{\ent}[1]{\left\lfloor #1\right\rfloor}
\newcommand{\card}[1]{\operatorname{Card}\pr{#1}}
\begin{document}

\maketitle

\begin{abstract}
 We show a deviation inequality  for
$U$-statistics of independent data taking values in a
separable Banach space which satisfies some smoothness assumptions.  We then
provide applications to  rates in the law of large numbers for  $U$-statistics, 
a Hölderian functional central limit theorem and a moment inequality for 
incomplete $U$-statistics.
\end{abstract}

\section{Deviation inequalities for $U$-statistics}

\subsection{Introduction and motivations}

Furnishing a bound for the probability that a random variable is bigger than
some fixed numbers plays a very important role in probability theory and its
applications. It can be used in order to control
the convergence rates in the law of large numbers, which can
translate into consistency rates for estimators. Moreover, tightness criterion
can be checked via such deviation inequalities.

Since weighted series of tails can be bounded by moments of a random variable,
a direct application of deviation inequalities can provide the
aforementioned results, in particular, without using truncation arguments.

We plan to furnish such inequalities for $U$-statistics, which are defined as
follows. Let $m$ be an integer, let  $\pr{\xi_i}_{i\in\Z}$ be
an i.i.d.\ sequence with values in a measurable space
$\pr{S,\Sca}$ and for $\gri=\pr{i_\ell}_{\ell=1}^m$ such
that $1\leq i_1<\dots<i_m$, let $h_{\gri}\colon S^m\to \B$
be measurable functions, where $\pr{\B,\norm{\cdot}_{\B}}$ is
a separable Banach space and $S^m$ is endowed with
the product $\sigma$-algebra of $\Sca$. The associated $U$-statistic
of order $m$ is defined by
\begin{equation}\label{eq:def_U_stat}
  U_{m,n}\pr{\pr{h_{\gri}}}=\sum_{
\gri\in \inc_n^m}
h_{\gri}\pr{\xi_{\gri}}, \quad n\geq m,
\end{equation}
where
\begin{equation}
 \inc_n^m=\ens{ \gri=\pr{i_\ell}_{\ell\in\intent{1,m}}, 1\leq 
i_1<i_2<\dots<i_m\leq
n},
\end{equation}
 for $k\leq m$,  $\intent{k,m}$ denotes the set $\ens{k,\dots,m}$ and 
$\xi_{\gri}=\pr{\xi_{i_1},\dots,\xi_{i_m}}$.

Such a general framework allows to consider classical $U$-statistics (when
$h_{\gri}$ is independent of $\gri$),
weighted $U$-statistics and incomplete $U$-statistics (see 
Subsection~\ref{subsec:moment_inequ_incomp} for a formal
definition).

We would like to bound the tails of  maxima of the norm
of $U_n$, that is, we would like to find a bound for
\begin{equation}
 \PP\pr{\max_{m\leq n\leq N}\norm{U_{m,n}\pr{\pr{h_{\gri}}}}_{\B}>t }
\end{equation}
in terms of the tail of a suitable random variable.

Exponential bounds have been investigated in the literature, see for 
instance \cite{MR1742613,MR1336800,MR1323145,MR2073426,MR1857312}.
The case of not necessarily independent random variables has also
been addressed in \cite{MR4550210,MR3769824,MR4059185,MR2797944}.
In this paper, we shall present inequalities
for $U$-statistics taking values in Banach spaces and such that the
tail of the random variables
$\norm{h_{\gri}\pr{\xi_{\gri} }}_{\B}$ decays
like a power function.

For real-valued martingales difference sequences,
it is known that we can control the tail of maxima of partial sums 
via the tail of maxima and the sum of conditional variances (see 
\cite{MR2021875}). More precisely, given a positive $q$, there exists a 
constant $C_q$ such that for each real-valued martingale difference 
sequence $\pr{D_i}_{i\geq 1}$ with respect to a filtration 
$\pr{\Fca_i}_{i\geq 0}$ and each positive $t$, 
\begin{multline}\label{eq:deviation_nagaev}
 \PP\pr{\max_{1\leq n\leq N}\abs{\sum_{i=1}^n D_i}>t  }\\
 \leq C_q\int_0^1 u^{q-1}\PP\pr{\max_{1\leq i\leq N}\abs{D_i}>tu}du+
 C_q\int_0^1 u^{q-1}\PP\pr{\pr{\sum_{i=1}^N\E{D_i^2\mid\Fca_{i-1}}}^{1/2}>tu}du.
\end{multline}
A Banach-valued version with not necessarily finite variance is given in 
Theorem~1.3 of \cite{MR4046858}. 
Such inequalities are very convenient to use, especially in the case where the 
random variables $\abs{D_i}$ have the same distribution. In this case, 
\eqref{eq:deviation_nagaev} gives for each $q>2$, 
\begin{equation*}
  \PP\pr{\max_{1\leq n\leq N}\abs{\sum_{i=1}^n D_i}>t\sqrt{N}  }
  \leq C_q\int_0^\infty\PP\pr{\abs{D_1}>tu}\min\ens{u^{q-1},u}du
  \leq C'_q\E{\min\ens{\frac{\abs{D_1}^q}{t^q},\frac{ D_1 ^2}{t^2}  }},
\end{equation*}
where $C'_q$ depends only on $q$ (Theorem~1.7 of \cite{MR4046858}, where a 
version for random variables with infinite variance is also presented). 
Then one can deduce convergence of series of the form 
$A:=\sum_{N=1}^\infty N^\alpha\PP\pr{\max_{1\leq n\leq N}\abs{\sum_{i=1}^n 
D_i}>N^\beta}$ by series involving only the tail of $\abs{D_1}$
and get thanks to the elementary inequality $\sum_{N=1}^\infty 2^{N 
p}\PP\pr{Y>2^N}\leq C_p\E{Y^p}$, an integrability condition on $D_1$ to get the 
convergence of the series $A$. This leads to the obtainment of 
convergence rates in the strong law of large numbers. Moreover, 
it is possible to check some tightness criteria in Hölder spaces for 
partial sum processes which are expressed in terms of tails (see equation (1.4) 
in \cite{MR3615086} and Proposition~1.1 in \cite{MR4294337}).

Motivated by these applications, we would like to formulate inequalities of the 
form \eqref{eq:deviation_nagaev} for $U$-statistics defined as in 
\eqref{eq:def_U_stat} and find bounds expressible as $\int_0^1u^{q-1}
\PP\pr{Y_k>tu}du$, where $Y_k$ are non-negative random variables related to the 
kernels $h_{\gri}$ and the random variables $\xi_1,\dots,\xi_m$. Under 
some conditions of degeneracy (see Definition~\ref{def:degenerated}), 
the considered $U$-statistic is a sum of martingale differences with respect 
to each index of summation. However, the term  corresponding to condition 
variances can only be expressed as 
a random variable conditioned with respect to some $\sigma$-algebra. This make 
the iteration of inequality \eqref{eq:deviation_nagaev} difficult and 
led us to control the conditional moments of a martingale difference sequence. 
Also, we will use reasoning by induction on the order $m$ of the $U$-statistic 
and it turns out that the consideration of Banach-valued random variables helps 
since the conditional variance term that arises can be viewed as a martingale 
difference sequence in an other Banach space (such an idea was used in 
\cite{giraudo2022deviation} for multi-indexed martingales).

Once the wanted deviation inequality is established, one can readily derive a 
moment inequality which is in the spirit of the one obtained in 
\cite{MR1734266}, but for real-valued $U$-statistics of order two. 
For Banach-valued random variables, a result has been obtained in
\cite{MR2294982}, where the upper bound is expressed in terms of
operator norm.

\subsection{A deviation inequality for $U$-statistics}

In order  to formulate deviation inequalities for
$\norm{U_n}_{\B}$ and $\max_{m\leq n\leq N}\norm{U_n}_{\B}$, one has to make
assumptions on the random variables $h_{\gri}\pr{\xi_{i_1},
 \dots,\xi_{i_m}}$, even if this random variable is centered. To see this,
 consider the case $\B=\R$ and $m=2$. If $h_{i_1,i_2}\pr{x,y}=x+y$ then
 $U_n=\pr{n-1}\sum_{i=1}^nX_i$ while with $h_{i_1,i_2}\pr{x,y}=xy$ we have
 $U_n=1/2\pr{\sum_{i=1}^nX_i}^2-\sum_{i=1}^n X_i^2$ hence the behavior of the
tails will change drastically.

In order to formulate general results, we will adopt the following definition.
\begin{Definition}\label{def:degenerated}
 We say that a function $h\colon S^m\to \B$ is degenerated with respect to the
 i.i.d.\ $S$-valued sequence $\pr{\xi_i}_{i\in\Z}$ if for each
 $\ell_0\in\intent{1,m}$,
\begin{equation}\label{eq:def_degenerated}
 \E{h\pr{\xi_{\intent{1,m}}} \mid  \xi_{\intent{1,m}\setminus 
\ens{\ell_0}}}=0,
\end{equation}
where 
\begin{equation}\label{eq:def_xi_J}
 \xi_{J}=\pr{\xi_j}_{j\in J}.
\end{equation}

 We say that a $U$-statistic of the form \eqref{eq:def_U_stat} is
 degenerated with
 respect to the sequence $\pr{\xi_i}_{i\in\Z}$ if all the functions
$h_{\gri}$, $\gri\in\inc^m:=\ens{\gri=\pr{i_\ell}_{\ell\in\intent{1,m}},1\leq 
i_1<\dots<i_m }$ are degenerated with
 respect to the sequence $\pr{\xi_i}_{i\in\Z}$.
\end{Definition}
However, degeneracy of $U_n$ does not always hold, like the previous example
$h_{i_1,i_2}\pr{x,y}=x+y$ shows. Nevertheless, it is possible to express
$U_n$ as sum of degenerated $U$-statistics of smaller order. Indeed, for a
subset $I$ of $\intent{1,m}$, let us denote by $\abs{I}$ the
cardinal of $I$ and define the function $h_{\gri}^I\colon
S^{\abs{I}}\to \B$ by
\begin{equation}\label{eq:def_hI}
 h_{\gri}^I\pr{ \pr{x_{i_u}}_{u\in I} }
 =\sum_{J:J\subseteq I}\pr{-1}^{\abs{I}-\abs{J}}
 \E{h_{\gri}\pr{ V^{I,J}\pr{\pr{x_{i_u}}_{u\in I} }
} },
\end{equation}
where the random vector $V^{I,J}\pr{\pr{x_{i_u}}_{u\in I}}$ belongs to $S^m$,
the $k$-th coordinates is $x_{i_u}$ if $k=i_u$ for some (hence exactly one)
$u\in J$ and $\xi_k$ if $k$ is not of this form. In this way, writing 
$\gri_J=\pr{i_\ell}_{\ell\in J}$ the equality
\begin{equation}\label{eq:fonctions_dans_Hoeffding_egalite_esp_cond}
  h_{\gri}^I\pr{ \pr{\xi_{i_u}}_{u\in I} }
  =\sum_{J:J\subseteq I}\pr{-1}^{\abs{I}-\abs{J}}
 \E{h_{\gri}\pr{ \xi_{\gri}
} \mid    \xi_{\gri_J} }
\end{equation}
holds almost surely, where $\E{\cdot\mid \xi_{\gri_\emptyset}}=\E{\cdot}$. 
Moreover, the equality
\begin{equation}
 \sum_{I\subset\intent{1,m}}h_{\gri}^I\pr{\xi_{\gri_I}
}=h_{\gri}\pr{\xi_{\gri}}
\end{equation}
takes places (to see this, one can switch the sums over $I$ and $J$ and use the
fact that $\sum_{I:J\subseteq I}\pr{-1}^{\abs{I}-\abs{J}}$ equals $0$ if $J$ is
not $\intent{1,m}$ and $1$ otherwise). As a consequence, the $U$-statistic
$U_n$ defined by \eqref{eq:def_U_stat} can be decomposed as
\begin{equation}\label{eq:decomp_Hoeffding}
 U_n=\sum_{I:I\subset\intent{1,m}}U_n^I,
\end{equation}
where $U_n^{\emptyset}=\sum_{\gri\in \inc_n^m}
\E{h_{\gri}\pr{\xi_{\intent{1,m}}}}$ and for
a non-empty $I\subset\intent{1,m}$,
\begin{equation}\label{eq:def_Ustats_dans_la_decomposition}
 U_n^{I}= \sum_{\gr{j_I}\in \inc_n^{\abs{I}}  }
 H^{I}_{\gr{j_I}}\pr{  \xi_{\gr{j_I}}   },
\end{equation}
\begin{equation}\label{eq:def_function_HI}
H^{I}_{\gr{j_I}}\pr{  \xi_{\gr{j_I}}   }
 = \sum_{\gr{i_{\intent{1,m}\setminus I  }
 }}h_{\sum_{\ell\in I}j_\ell \gr{e_\ell}+\sum_{\ell'\in \intent{1,m}\setminus I
}i_{\ell'}\gr{e_{\ell'}}}^I\pr{  \xi_{\gr{j_I}}   },
\end{equation}
$\gr{e_\ell}$ is a vector whose entry $\ell$ is one and the others zero, 
and the sum runs over the $\gr{i_{\intent{1,m}\setminus I  }
 }$ such
that
\begin{equation}
\sum_{\ell\in I}
 j_\ell \gr{e_\ell}+\sum_{\ell'\in \intent{1,m}\setminus I
}i_{\ell'}\gr{e_{\ell'}}\in \inc_{n}^{\abs{I}}.
\end{equation}
In this way, the $U$-statistic $ U_n^{I}$ has order $\abs{I}$ and is
degenerated.
When $h_{\gri}$ does not depend on $\gri$, this was done in
Theorem~9.1 in \cite{MR3087566}.

For $m=2$, the decomposition reads
\begin{equation}
 U_n=U_n^{\emptyset}+U_n^{\ens{1}}+U_n^{\ens{2}}+U_n^{\ens{1,2}},
\end{equation}
where
\begin{equation}
 U_n^{\emptyset}=\sum_{1\leq i_1<i_2\leq n}\E{h_{i_1,i_2}\pr{\xi_1,\xi_2}},
\end{equation}
\begin{equation}
U_n^{\ens{1}}=\sum_{i_1=1}^{n-1}  \sum_{j_2=i_1+1}^n
\E{h_{i_1,j_2}\pr{\xi_{i_1},\xi_{j_2}}\mid\xi_{i_1} },
\end{equation}
\begin{equation}
U_n^{\ens{2}}=\sum_{i_2=2}^{n}  \sum_{j_1=1}^{i_2-1}
\E{h_{j_1,i_2}\pr{\xi_{j_1},\xi_{i_2}}\mid\xi_{i_2} },
\end{equation}
\begin{multline}
U_n^{\ens{1,2}}=\sum_{1 \leq i_1<i_2\leq n}\left(
h_{i_1,i_2}\pr{\xi_{i_1},\xi_{i_2}}-
\E{h_{i_1,i_2}\pr{\xi_{i_1},\xi_{i_2}}\mid\xi_{i_1}
}+\right.\\
\left.
-\E{h_{i_1,i_2}\pr{\xi_{i_1},\xi_{i_2}}\mid\xi_{i_1} }+
\E{h_{i_1,i_2}\pr{\xi_{i_1},\xi_{i_2}} }\right).
\end{multline}

Note that degenerated $U$-statistics enjoy a martingale property. It is known
that the geometry of
the involved Banach space plays a important role in the derivation of moments
and deviation inequalities.

\begin{Definition}
 Let $\pr{\B,\norm{\cdot}_{\B}}$ be a separable Banach space. We say that $\B$
is $r$-smooth for $1<r\leq 2$ if there exists an equivalent norm
$\norm{\cdot}'_{\B}$ on $\B$ such that
\begin{equation}
\sup_{t>0}\sup_{x,y\in\B,\norm{x}'_{\B}=
\norm{y}'_{\B}=1,}\frac{\norm{x+ty}'_{\B}
+\norm{x-ty}'_{\B}-2}{
t^r}<\infty.
\end{equation}
\end{Definition}

By \cite{MR0407963}, we know that if $\B$ is a separable $r$-smooth Banach
space, then
there exists a constant $C$ such that for each martingale
difference sequence $\pr{D_i}_{i\geq 1}$
with values in $\B$ and each $n$,
\begin{equation}
 \E{\norm{\sum_{i=1}^nD_i}_{\B}^r}\leq C\sum_{i=1}^n\E{\norm{D_i}_{\B}^r}.
\end{equation}
By definition, an $r$-smooth Banach space is also $p$-smooth for $1<p\leq r$,
hence it is possible to define
\begin{equation} \label{eq:definition_constant_Banach_space}
 C_{p,\B}:=\sup_{n\geq 1}\sup_{\pr{D_i}_{i=1}^n \in \Delta_n}
\frac{\E{\norm{\sum_{i=1}^nD_i}_{\B}^p} }{\sum_{i=1}^n\E{\norm{D_i}_{\B}^p}},
\end{equation}
where the $\Delta_n $ denotes the set of the  martingale difference sequences
$\pr{D_i}_{i=1}^n$ such that $\sum_{i=1}^n \norm{D_i}_{\B}^p$ is not
identically $0$.

We will now state a general deviation
inequality for Banach valued $U$-statistics. The obtained bound involves 
$2^m$ terms. For each subset $J$ of $\intent{1,m}$, we sum over the indices in 
$J$ a function of the tail of sums of conditional $p$-th moments of 
$\norm{h_{\gri}\pr{\xi_{\intent{1,m}}}}_{\B}$ over the coordinates in 
$\intent{1,m}\setminus J$.

For a proper non-empty subset $J$ of $\intent{1,m}$, we define the subset of $\N^m$, denoted $\N^{J,m}$, by $
\N^{J,m}=\ens{
\sum_{j\in J}v_j\gr{e_j}, v_j\in\N}$, where $\gr{e_j}$ is the element of $\N^m$ 
having $1$ at coordinate $j$ and $0$ at the others.  
For a subset $J$ of $\intent{1,m}$, recall the notation  
$\xi_J$ given by \eqref{eq:def_xi_J}.

The key deviation inequality on which all the results of the paper rest reads 
as follows.
\begin{Theorem}\label{thm:deviation_inequality_Ustats}
Let $\pr{\B,\norm{\cdot}_{\B}}$ be an $r$-smooth Banach space for some 
$r\in (1,2]$. Let $\pr{\xi_i}_{i\geq 1}$ be an i.i.d.\ sequence
taking values in a measurable space $\pr{S,\Sca}$, let
$h_{\gri}\colon S^m\to \B$ be degenerated with
respect to $\pr{\xi_i}_{i\geq 1}$, that is, such that for each
$\ell_0\in\intent{1,m}$,
\begin{equation}
\E{h_{\gri}\pr{\xi_{\intent{1,m}}}\mid 
 \xi_{\intent{1,m} \setminus \ens{\ell_0}} }=0.
\end{equation}
The following inequality takes place for each $p\in (1,r]$  and each positive $t$ and $q$:
\begin{multline}\label{eq:deviation_inequality_Ustats}
\PP\pr{\max_{m\leq n\leq 
N}\norm{\sum_{\gri\in\inc_n^m}h_{\gri}\pr{\xi_{\gri}} 
}_{\B}>t }\\
\leq K\pr{m,p,q,\B}\sum_{\gri\in\inc_N^m}
\int_0^1u^{q-1}\PP\pr{\norm{h_{\gri}\pr{\xi_{\intent{1,m}}}}_{\B} >tu} du\\
+K\pr{m,p,q,\B}
\sum_{\emptyset \subsetneq J\subsetneq \intent{1,m}}\sum_{\gr{i_J}\in\N^{J,m} 
}\int_0^1 u^{q-1}
\PP\pr{\pr{ \sum_{\gr{i_{J^c}}: \gr{i_J}+\gr{i_{J^c}}\in \inc_N^m } \E{ \norm{ 
h_{\gr{i_J}+\gr{i_{J^c}}}\pr{\xi_{\intent{1,m}}}   }_{\B}^p  \mid \xi_J  }   
}^{1/p}>tu}du\\
+K\pr{m,p,q,\B}t^{-q}\pr{\sum_{\gri\in\inc_N^m }
\E{\norm{h_{\gri}\pr{\xi_{\intent{1,m}}}}^p_{\B}}  }^{q/p},
\end{multline}
where $K\pr{m,p,q,\B}$ depends only on $m$, $p$, $q$ and $\B$.
\end{Theorem}

For $U$-statistics of order two, Theorem~\ref{thm:deviation_inequality_Ustats} 
reads as follows: if $\pr{\xi_i}_{i\in\Z}$ is an i.i.d.\ sequence taking 
values in
$\pr{S,\Sca}$ and $h_{i_1,i_2}\colon S^2\to\B$ is such that for each $i_1<i_2$,
\begin{equation}\label{eq:deg_moment_ineq_m=2}
\E{h_{i_1,i_2}\pr{\xi_1,\xi_2}\mid\xi_1}=\E{h_{i_1,i_2}\pr{\xi_1,\xi_2}\mid\xi_2
} = 0 ,
\end{equation}
then 
\begin{multline}\label{eq:deviation_inequality_Ustats_ordre_2}
\PP\pr{
\max_{2\leq n\leq N}\norm{\sum_{1\leq i_1< i_2\leq 
n}h_{i_1,i_2}\pr{\xi_{i_1},\xi_{i_2}}}_{\B}
>t}\\   \leq  K\pr{2,p,q,\B}\sum_{1\leq i_1<i_2\leq N}
\int_0^1 u^{q-1} \PP\pr{\norm{h_{i_1,i_2}\pr{\xi_1,\xi_2}}_{\B} >tu}du\\
+ K\pr{2,p,q,\B}\sum_{i_1=1}^{N-1}\int_0^1 u^{q-1} \PP\pr{\pr{
\sum_{i_2=i_1+1}^N\E{\norm{h_{i_1,i_2}\pr{\xi_1,\xi_2}}^p_{\B}\mid\xi_1}}^{1/p} 
>tu}du\\
+ K\pr{2,p,q,\B}\sum_{i_2=2}^{N}\int_0^1 u^{q-1} \PP\pr{\pr{
\sum_{i_1=1}^{i_2-1}\E{\norm{h_{i_1,i_2}\pr{\xi_1,\xi_2}}^p_{\B}\mid\xi_2}}^{1/p
} >tu}du\\
+ K\pr{2,p,q,\B}t^{-q}\pr{\sum_{1\leq i_1<i_2\leq 
N}\E{\norm{h_{i_1,i_2}\pr{\xi_1,\xi_2}}^p_{\B} }}^{q/p}.
\end{multline}

When the kernel functions $h_{\gri}$ are 
independent of the index $\gri$, Theorem~\ref{thm:deviation_inequality_Ustats} 
admits 
the following simpler form.
\begin{Corollary}\label{cor:meme_noyau}
Let $\pr{\B,\norm{\cdot}_{\B}}$ be an $r$-smooth Banach space for some 
$r\in (1,2]$. Let $\pr{\xi_i}_{i\geq 1}$ be an i.i.d.\ sequence taking values
in
a measurable space $\pr{S,\Sca}$, let $h \colon S^m\to \B$ be degenerated with 
respect to $\pr{\xi_i}_{i\geq 1}$, that is, for each $\ell_0\in\intent{1,m}$,  
\begin{equation}\label{eq:deg_meme_noyau}
\E{h \pr{\xi_{\intent{1,m}}}\mid  \xi_{ \intent{1,m} 
\setminus \ens{\ell_0}} }=0.
\end{equation}
The following inequality takes place for each $p\in (1,r]$  and each positive $t$ and $q$:
\begin{multline}\label{eq:deviation_inequality_Ustats_meme_noyau}
\PP\pr{\max_{m\leq n\leq 
N}\norm{\sum_{\gri\in\inc_n^m}h\pr{\xi_{\gri}} }_{\B}>t }\\
\leq K\pr{m,p,q,\B} N^m
\int_0^1u^{q-1}\PP\pr{\norm{h \pr{\xi_{\intent{1,m}}}}_{\B} >tu} du\\
+K\pr{m,p,q,\B}
\sum_{\emptyset \subsetneq J\subsetneq \intent{1,m}} N^{\abs{J}}\int_0^1 u^{q-1}
\PP\pr{N^{\frac{m-\abs{J}}{p}}\pr{   \E{ \norm{ h \pr{\xi_{\intent{1,m}}}   
}_{\B}^p  \mid \xi_\ell,\ell\in J  }   }^{1/p}>tu}du\\
+K\pr{m,p,q,\B}t^{-q}N^{mq/p}\pr{ 
\E{\norm{h \pr{\xi_{\intent{1,m}}}}^p_{\B}}  }^{q/p},
\end{multline}
where $K\pr{m,p,q,\B}$ depends only on $m$, $p$, $q$ and $\B$.
\end{Corollary}

In order to address the non-necessarily degenerated case, some assumptions and definitions are required.
\begin{Definition}
We say that the kernel $h\colon S^m\to\B$ is symmetric if for each 
bijection $\sigma\colon\intent{1,m}\to \intent{1,m}$ and each $x_1,\dots,x_m\in S$, 
\begin{equation}
h\pr{x_{\sigma\pr{1}},\dots,x_{\sigma\pr{m}}}=h\pr{x_1,\dots,x_m}.
\end{equation}
\end{Definition}

\begin{Definition}
Let $d\in  \intent{1,m }$ and let $\pr{\xi_i}_{i\in\Z}$ be
an i.i.d.\ sequence taking values in $S$. We say that the
symmetric kernel $h\colon S^m\to\B$ is degenerated of
order $d$ with respect to $\pr{\xi_i}_{i\in\Z}$ if 
\begin{equation}
 \E{h\pr{\xi_{\intent{1,m}}} \mid \xi_{\intent{1,d-1 } } }  =0 
\mbox{ a.s. and 
 }
 \E{\norm{\E{h\pr{\xi_{\intent{1,m}}} \mid \xi_{\intent{1,d} } 
}}_{\B}}\neq 
0.
\end{equation}
\end{Definition}

If $h$ is degenerated of order $d-1$ with respect to $\pr{\xi_i}_{i\in\Z}$, 
it is possible to write 
\begin{equation}\label{eq:decomposition_somme_U_stat_deg}
\sum_{\gri\in\inc^m_n}h\pr{\xi_{\gri}}
=\binom{n}{m}\sum_{c=d}^m\binom mc\binom nc^{-1}
\sum_{\gri\in\inc^c_n 
}h^{\pr{c}}\pr{\xi_{\gri}},
\end{equation}
where 
\begin{equation}\label{eq:def_hc}
h^{\pr{c}}\pr{x_1,\dots,x_c}=\sum_{k=0}^c\pr{-1}^{c-k}
\sum_{\gri\in\inc^k_c } 
\E{h\pr{x_{\gri},\xi_{\intent{1,m-k} }}}.
\end{equation}
In other words, the $U$-statistic of kernel $h$ can be written as a
weighted sum of $\pr{m-d+1}$ $U$-statistics, each of them being
degenerated.

We are now in position to provide a bound for $U$-statistics having a symmetric 
but not necessarily degenerated kernel. 

\begin{Corollary}\label{cor:ineg_deviation_deg_ordre_d_sym}
 Let $\pr{\B,\norm{\cdot}_{\B}}$ be an $r$-smooth Banach space for some 
$r\in (1,2]$. Let $\pr{\xi_i}_{i\geq 1}$ be an i.i.d.\ sequence taking values
in
a measurable space $\pr{S,\Sca}$, let $h \colon S^m\to \B$ be a symmetric 
function with is degenerated of order $d$ with 
respect to $\pr{\xi_i}_{i\geq 1}$. Define 
\begin{equation}
 H_p:=\max_{k\in\intent{0,m}} \pr{\E{ 
\norm{h\pr{\xi_{\intent{1,m}}}   }_{\B}^p 
\mid\xi_{\intent{1,k}}}  }^{1/p}.
\end{equation}
The following inequality takes place for each 
$p\in (1,r]$  and each positive $t$ and $q$:
\begin{multline}\label{eq:ineg_deviation_deg_ordre_d_sym}
 \PP\pr{\max_{m\leq n\leq N}\norm{\sum_{\gri\in\inc_n^m}h\pr{\xi_{\gri}} 
}_{\B}>tN^{m-d+\frac dp} }\\
\leq K\pr{m,p,q,\B} \sum_{j=0}^m 
N^j \int_0^1u^{q-1}\PP\pr{    H_p>tN^{\pr{\max\ens{d,j}-d}\frac{p-1}p+\frac jp 
}   u }du.
\end{multline}
\end{Corollary}
Notice that the term of index $j\in\intent{1,d}$ in the right hand side of  
\eqref{eq:ineg_deviation_deg_ordre_d_sym}  
the term is of the form $N^j \int_0^1u^{q-1}\PP\pr{    
H_p>tN^{ \frac jp}   u }du$, which is independent of $N$ for $j=0$ and for 
$j\in\intent{d+1,m}$, the corresponding term is 
$N^j \int_0^1u^{q-1}\PP\pr{    
H_p>tN^{j-d+d/p}   u }du$. All these terms bring a different and not easily 
comparable contribution.
\subsection{A moment inequality}

One can deduce from Theorem~\ref{thm:deviation_inequality_Ustats} a moment inequality.

\begin{Corollary}\label{cor:moment_inequality}
Let $\pr{\B,\norm{\cdot}_{\B}}$ be an $r$-smooth Banach space for some 
$r\in (1,2]$. Let $\pr{\xi_i}_{i\geq 1}$ be an i.i.d.\ sequence taking values in
a measurable space $\pr{S,\Sca}$, let $h_{\gri}\colon S^m\to \B$ be
degenerated with respect to $\pr{\xi_i}_{i\geq 1}$ and
such that for each $\ell_0\in\intent{1,m}$,  
\begin{equation}\label{eq:deg_moment_ineq}
\E{h_{\gri}\pr{\xi_{\intent{1,m}}}\mid  \xi_{ \intent{1,m} 
\setminus \ens{\ell_0}} }=0.
\end{equation}
For each $q\geq p$, the following inequality holds:
\begin{multline}\label{eq:moment_ineq_degenerated}
\E{\max_{m\leq n\leq N}\norm{\sum_{\gri
\in\inc_n^m}h_{\gri}\pr{\xi_{\gri}} }_{\B}^q }\leq 
K\pr{m,p,q,\B}\sum_{\gri\in\inc_N^m}
\E{\norm{h_{\gri}\pr{\xi_{\intent{1,m}}}}_{\B}^q }\\
+K\pr{m,p,q,\B}
\sum_{\emptyset \subsetneq J\subsetneq \intent{1,m}}\sum_{\gr{i_J}\in\N^J } 
\E{\pr{ \sum_{\gr{i_{J^c}}: \gr{i_J}+\gr{i_{J^c}}\in \inc_N^m } \E{ \norm{ 
h_{\gr{i_J}+\gr{i_{J^c}}}\pr{\xi_{\intent{1,m}}}   }_{\B}^p  \mid 
\xi_\ell,\ell\in J  }   
}^{q/p} }\\
+K\pr{m,p,q,\B} \pr{\sum_{\gri\in\inc_N^m }
\E{\norm{h_{\gri}\pr{\xi_{\intent{1,m}}}}^p_{\B}}  }^{q/p}.
\end{multline}
\end{Corollary}
When $m=2$, the assumption \eqref{eq:deg_moment_ineq} admits the simpler form
\eqref{eq:deg_moment_ineq_m=2} and \eqref{eq:moment_ineq_degenerated} reads as
follows:
\begin{multline}\label{eq:moment_ineq_degenerated_ordre_2}
\E{
\max_{2\leq n\leq N}\norm{\sum_{1\leq i<\leq j\leq 
n}h_{i,j}\pr{\xi_i,\xi_j}}_{\B}^q
}\leq K\pr{2,p,q,\B}\sum_{1\leq i<j\leq N}
\E{\norm{h_{i,j}\pr{\xi_1,\xi_2}}_{\B}^q}\\
+ K\pr{2,p,q,\B}\sum_{i=1}^{N-1}\E{\pr{
\sum_{j=i+1}^N\E{\norm{h_{i,j}\pr{\xi_1,\xi_2}}^p_{\B}\mid\xi_1}}^{q/p}}\\
+ K\pr{2,p,q,\B}\sum_{j=2}^{N}\E{\pr{
\sum_{i=1}^{j-1}\E{\norm{h_{i,j}\pr{\xi_1,\xi_2}}^p_{\B}\mid\xi_2}}^{q/p} }\\
+ K\pr{2,p,q,\B} \pr{\sum_{1\leq i<j\leq 
N}\E{\norm{h_{i,j}\pr{\xi_1,\xi_2}}^p_{\B} }}^{q/p}.
\end{multline}
Notice also that for $q=p$, \eqref{eq:moment_ineq_degenerated} reads as 
follows: for $1<p\leq r$, 
\begin{equation}\label{eq:moment_ineq_degenerated_s=p}
\E{\max_{m\leq n\leq N}\norm{\sum_{\gri\in\inc_n^m}h_{\gri}\pr{\xi_{\gri} 
}}_{\B}^p }\\
\leq 
K\pr{m,p,\B}\sum_{\gri\in\inc_N^m}
\E{\norm{h_{\gri}\pr{\xi_{\intent{1,m}}}}_{\B}^p } .  
\end{equation}
In particular, if $h_{\gri}=h$ and \eqref{eq:deg_meme_noyau} holds, 
then 
\begin{equation}\label{eq:moment_ineq_degenerated_s=p_meme_noyau}
\E{\max_{m\leq n\leq N}\norm{\sum_{\gri\in\inc_n^m}h \pr{\xi_{ \gri}  } 
}_{\B}^p }
\leq 
K\pr{m,p,\B} N^{m}
\E{\norm{h\pr{\xi_{\intent{1,m}} }}_{\B}^p } .  
\end{equation}

Moreover, the deviation inequality \eqref{eq:deviation_inequality_Ustats} 
allows to derive a Rosenthal type inequality for weak $\el^q$-moments: define 
for a random variable $Y$ taking values in a Banach space 
$\pr{\B,\norm{\cdot}_{\B}}$, and $q>1$ the quantity
\begin{equation}
 \norm{Y}_{\B,q,w}:=\pr{
 \sup_{t>0}t^q\PP\pr{\norm{Y}_{\B}>t}
 }.
\end{equation}
Then applying \eqref{eq:deviation_inequality_Ustats} 
with $q$ replaced by $q+1$ gives (under the assumptions of 
Theorem~\ref{thm:deviation_inequality_Ustats})
\begin{multline}
 \norm{\max_{m\leq n\leq N}\norm{\sum_{\gri
\in\inc_n^m}h_{\gri}\pr{\xi_{\gri}} }_{\B} }_{\R,q,w}^q\leq 
K\pr{m,p,q,\B}\sum_{\gri\in\inc_N^m}
 \norm{h_{\gri}\pr{\xi_{\intent{1,m}}}}_{\B,q,w}^q 
 \\
+K\pr{m,p,q,\B}
\sum_{\emptyset \subsetneq J\subsetneq \intent{1,m}}\sum_{\gr{i_J}\in\N^J } 
\norm{\pr{ \sum_{\gr{i_{J^c}}: \gr{i_J}+\gr{i_{J^c}}\in \inc_N^m } \E{ \norm{ 
h_{\gr{i_J}+\gr{i_{J^c}}}\pr{\xi_{\intent{1,m}}}   }_{\B}^p  \mid 
\xi_\ell,\ell\in J  }   
}^{1/p} }_{\R,q,w}^q\\
+K\pr{m,p,q,\B} \pr{\sum_{\gri\in\inc_N^m }
\E{\norm{h_{\gri}\pr{\xi_{\intent{1,m}}}}^p_{\B}}  }^{q/p}.
\end{multline}

Let us connect these results with existing ones in the literature. 
\begin{itemize}
 \item In \cite{MR4040993}, a Rosenthal type inequality 
 for Hermitian matrix-valued $U$-statistics was obtained. The presented bound for 
the moments of such a $U$-statistic has explicit constant, 
and the terms in the spirit of second   
the right hand side of \eqref{eq:moment_ineq_degenerated_ordre_2} involves 
the moment of order $q$ of the maximum over $i$ instead of 
the sum of moments. However, our result deals with vector-valued random 
variables and $U$-statistics of any order.
\item In \cite{MR2294982}, a moment inequality for 
$U$-statistics taking values in a separable Banach space 
is given. On the one hand, the assumption on the Banach space (none in 
Theorem~1, type $2$ in Theorem~2) is less restrictive than ours and  the 
constants 
are explicit. The bounds are formulated in terms of moments of some random 
variable which are less explicit those involved in 
the right hand side of \eqref{eq:moment_ineq_degenerated}. 
Moreover, we deal with the case where $h_{\gri}\pr{\xi_{\intent{1,m}}}$ 
does not necessarily admit a moment of order $2$ hence our result can be 
viewed as a complement of that of \cite{MR2294982}. A moment inequality
in the same spirit has also been obtained in \cite{MR1857312}.
\item In \cite{MR1955347}, a moment inequality for $U$-statistics
having non-negative kernels is established. As mentioned by the
authors in Remark~6, page 982, a Rosenthal type inequalities can be
derived but it does not seem to be easily comparable with
ours.
\item Rosenthal type inequalities were also established in
\cite{MR1734266}. They concern the real-valued case and
$U$-statistics of order two, while our inequality is for Banach
valued $U$-statistics of arbitrary order. However, they do not assume
that the random variables $\pr{\xi_i}_{i\geq 1}$ have the same
distribution.
\end{itemize}

\section{Applications}

\subsection{Rates in the law of large numbers for complete 
$U$-statistics}\label{subsec:LGN_comp_Ustats}

It is known that if $h\colon S^m\to \R$ and $\pr{\xi_{i}}_{i\geq 1}$ is an 
i.i.d.\ sequence for which $\E{\abs{h\pr{\xi_{\intent{1,m}} }}}<\infty$, then 
$U_{m,n}\pr{h}\to 0$ a.s. 
(see \cite{MR0026294}). Convergence under the assumption of finiteness of 
higher moments of $h\pr{\xi_{\intent{1,m}}}$ has been 
investigated in 
\cite{MR0881248,MR1054394,MR0336788,MR1227625,MR0358948,MR1400594}.

The first result of this subsection is a sufficient condition for
the Marcinkiewicz strong law of large numbers for Banach valued
degenerated $U$-statistics.

\begin{Theorem}\label{thm:fonction_maximale_Ustat}
 Let $m\geq 2$ be an integer, let $\pr{\xi_i}_{i\geq 1}$ be
 an i.i.d.\ sequence
taking values in a measurable space $\pr{S,\Sca}$, and let 
 $\pr{\B,\norm{\cdot}_{\B}}$ be an $r$-smooth separable Banach space 
 for $r\in (1,2]$. Let $h\colon S^m\to\B$ be a measurable 
 function such that for each $\ell_0\in\intent{1,m}$, 
 \begin{equation}\label{eq:deg_fct_max}
  \E{h\pr{\xi_{\intent{1,m}}}\mid  \xi_{
  \intent{1,m}\setminus\ens{\ell_0}
  } }=0.
 \end{equation}
 For each $1<p< r$, there exists a constant $K\pr{p,\B}$ depending only on $p$ 
and $\B$ such that 
 \begin{equation}\label{eq:fonction_maximale_Ustat}
  \sup_{t>0}t^p\PP\pr{
  \sup_{n\geq m}\frac{1}{n^{m/p }}\norm{\sum_{\gri
  \in\inc^m_n}h\pr{\xi_{\gri} } }_{\B}>t}  
  \leq K\pr{p,\B}\E{\norm{h\pr{\xi_{\intent{1,m}}}}_{\B}^p}.
 \end{equation}
 Moreover, 
 \begin{equation}
 \lim_{n\to\infty}\frac{1}{n^{m/p }}\norm{\sum_{\gri
  \in\inc^m_n}h\pr{\xi_{\gri} } }_{\B}=0\mbox{ a.s.}
 \end{equation}
\end{Theorem}

Such a result was known when $m=2$ and $\B=\R$ (see 
Proposition 1.2 in \cite{MR4243516}). Notice that we do not need
symmetry of the kernel $h$. Note that this is not a direct application of the 
established deviation inequality.

We now present some results on the rates in the strong law of large numbers, 
which can be derived from the deviation inequality we propose. 

\begin{Theorem}\label{thm:sup_Ustats}
 Let $m\geq 2$ be an integer, let $\pr{\xi_i}_{i\geq 1}$ be an
 i.i.d.\ sequence
taking values in a measurable space $\pr{S,\Sca}$, and let 
 $\pr{\B,\norm{\cdot}_{\B}}$ be an $r$-smooth separable Banach space 
 for $r\in (1,2]$. Let $h\colon S^m\to\B$ be a measurable symmetric
 function which is degenerated of order $d$. Define for $j\in\intent{1,m}$ 
 the random variable 
 \begin{equation}
 H_{j,r}:=\pr{\E{ \norm{h\pr{\xi_{\intent{1,m}}}}_{\B}^r \mid\xi_{\intent{1,j}}
}}^{1/r}.
 \end{equation}
 Suppose that $0<\alpha<\pr{r-1}d/r$
 and that for each $j$, 
 \begin{equation}\label{eq:def_Hjr}
H_{j,r}\in\mathbb{L}^{q\pr{d,j,\gamma,r}} , \mbox{ where }
q\pr{d,j,\gamma,r}:=\frac{\gamma+j+1}{\max\ens{d,j}\frac{r-1}r-\alpha+j
/r}  .
 \end{equation}
 Then for each positive $\eps$, 
 \begin{equation}
 \sum_{N=1}^\infty N^\gamma\PP\pr{\sup_{n\geq N}n^\alpha 
 \frac 1{\binom nm}\norm{U_{m,n}\pr{h}}_{\B}>\eps }<\infty,
 \end{equation}
 where 
 \begin{equation}
  U_{m,n}\pr{h}=\sum_{\gri
  \in\inc^m_n}h\pr{\xi_{\gri} }.
 \end{equation}
\end{Theorem}
Kokic obtained in \cite{MR0903815} a result in this spirit. Our result can be 
viewed as an extension in several directions. First, we consider the case where 
the random variables defined \eqref{eq:def_Hjr} may have different degrees of 
integrability (see Example~1 page 285 in \cite{MR1607435} for examples 
of kernels where $H_{j,r}$ can have prescribed integration degrees)
), while Kokic only makes an assumption on the moments of 
$h\pr{\xi_{\intent{1,m}}}$. Second, we address the case of Banach-valued 
$U$-statistics, while the result of \cite{MR0903815} is for real valued 
kernels. Finally, if we compare the results on a common setting ($\B=\R$, $r=2$ 
and only the assumption that $h\pr{\xi_{\intent{1,m}}}\in\el^q$ for some $q\geq 
2$), we can take $\gamma=q\pr{d/2-\alpha}-1$ whereas in Theorem~1 of 
\cite{MR0903815}, we can only take $\gamma=q\pr{d/2-\alpha}-1-\eta$ for some 
positive $\eta$. 
\subsection{Functional central limit theorems in Hölder spaces}

Given a $U$-statistic with fixed kernel $h\colon S^m\to \R$, it is possible to associate a partial sum process by defining 
\begin{equation}
\mathcal{U}_{n,h}\pr{t}=\sum_{\gri\in\inc^m_{\ent{nt}}}
h\pr{\xi_{\gri}}, \quad t\in [0,1],
\end{equation}
 where for $x\in \R$, $\ent{x}$ is the unique integer satisfying $\ent{x}\leq x<\ent{x}+1$.
 In \cite{MR740907}, the convergence in distribution in the 
 Skorohod space $D[0,1]$ of the process $\pr{n^{-r/2}
 U_{[n\cdot ]}}_{n\geq r}$ is studied. In Corollary~1, it is shown that if $U_n$ is degenerated 
 of order $d$, $d\in\intent{2,m}$, then $\pr{n^{-d/2} U_{[n\cdot 
]}}_{n\geq r}$ 
 converges in distribution to a process $I_d\pr{h_d}$ symbolically defined as 
 \begin{equation}\label{eq:def_integrale_Id}
 I_d\pr{h_d}\pr{t}=\int\dots\int h_d\pr{x_1,\dots,x_d}\mathbf 1_{[0,t]}\pr{u_1}
 \dots \mathbf 1_{[0,t]}\pr{u_d}W\pr{dx_1,du_1}\dots W\pr{dx_d,du_d},
 \end{equation}
where $W$ denotes the Gaussian measure (see the Appendix~A.1 and A.2 of the 
paper \cite{MR740907}). For $i=2$, the limiting process admits the 
expression $\sum_{j=1}^{+\infty} \lambda_j\pr{B_j^2\pr{t}-t}$, where 
$\pr{B_j\pr{\cdot}}_{j\geq 1}$ are independent standard Brownian motions and 
$\sum_{j=1}^{ \infty} \lambda_j^2$ is finite.
Notice that for each $\alpha\in \pr{0,1/2}$, the trajectories of the limiting 
process $I_d\pr{h_d}\pr{\cdot}$ are almost surely $\alpha$-Hölder continuous 
but of course, those of $t\mapsto \mathcal{U}_{n,h}\pr{t}$ are not. For this 
reason, we have to consider a linearly interpolated version of $\Uca_{n,h}$, 
denoted by $\Uca_{n,h}^{\operatorname{pl}}$ and defined by 
\begin{equation}\label{eq:definition_processus_sommes_partielles}
 \Uca_{m,n,h}^{\operatorname{pl}}\pr{t}=\begin{cases}
\sum_{\pr{i_\ell}_{\ell\in\intent{1,m}}\in\inc^m_{k}}
h\pr{\xi_{i_1},\dots,\xi_{i_m}}&\mbox{if }t=\frac kn \mbox{ for some 
}k\in\intent{0,n}\\
\mbox{linear interpolation}&\mbox{on }\pr{\frac 
kn,\frac{k+1}n},k\in\intent{0,n-1}.
                                \end{cases}
\end{equation}
Such a process has 
path in Hölder spaces. Therefore, the study of the limiting behavior of (an 
appropriately centered and normalized version of) 
$\pr{ \Uca_{m,n,h}^{\operatorname{pl}}\pr{t}}_{n\geq m}$ in Hölder spaces can 
be considered. We denote by $\Hca_\alpha$ the space of Hölder continuous 
functions on $[0,1]$, that is, the set of functions $x\colon [0,1]\to 
\R$ such that 
\begin{equation}
\norm{x}_\alpha:=\abs{x\pr{0}}+ 
\sup_{s,t\in[0,1],s<t}\frac{\abs{x\pr{s}-x\pr{t}}}{\pr{s-t }^\alpha}<\infty.
\end{equation}
The space $\Hca_\alpha$ endowed with $\norm{x}_\alpha$ is not separable and it 
will be more convenient to work with the separable subspace                     
\begin{equation}
 \Hca_0:=\ens{x\in\Hca_\alpha, \lim_{\delta\to 
0}\sup_{s,t\in[0,1],0<s-t<\delta}\frac{\abs{x\pr{s}-x\pr{t}}}{\pr{s-t }^\alpha} 
=0}.
\end{equation}

It has been shown in \cite{MR2000642} that if $\alpha\in\pr{0,1/2}$, 
$\pr{\xi_i}_{i\geq 1}$ is a real-valued i.i.d.\ centered sequence having unit 
variance and $W_n$ is the random function interpolating the points $\pr{0,0}$ 
and $\pr{k/n,n^{-1/2}\sum_{i=1}^k \xi_i}, k\in\intent{1,n}$, the convergence of 
$\pr{W_n}_{n\geq 1}$ in $\Hca_{\alpha}^o$ is equivalent to 
\begin{equation}\label{eq:CNS_PI_holderien_suites_iid}
 \lim_{t\to\infty}t^{p\pr{\alpha}}\PP\pr{\abs{ \xi_1 } >t   }=0,
\end{equation}
where $p\pr{\alpha}=\pr{1/2-\alpha}^{-1}$. Note that $p\pr{\alpha}$ is 
strictly bigger than $2$ hence \eqref{eq:CNS_PI_holderien_suites_iid} 
is more restrictive than having moments of order $2$. Moreover, the closer 
$\alpha$  to the critical value is, the more restrictive
\eqref{eq:CNS_PI_holderien_suites_iid}  is. 
\begin{Theorem}\label{thm:WIP_Holder_Ustats}
 Let $m\geq 2$, let $\pr{S,\Sca}$ be a measurable space,  let $h\colon S^m\to 
\R$ be a symmetric measurable function and let $\pr{\xi_i}_{i\geq 1}$ be a 
i.i.d.\ sequence taking values in $S$. Let $\alpha\in\pr{0,1/2}$ and 
$p\pr{\alpha}=\pr{1/2-\alpha}^{-1}$. Suppose that $h$ is degenerated of order 
$d$ 
and that 
\begin{equation}\label{eq:cond_suffisante_PI_holderien_Ustats}
 \lim_{t\to\infty}t^{p\pr{\alpha}}\PP\pr{\abs{h\pr{\xi_{\intent{1,m}}} }>t   
}=0.
\end{equation}
Then the following convergence in distribution in $\Hca_{\alpha}^o$ takes place:
\begin{equation}
 \pr{\frac{1}{n^{m-d/2}} \Uca_{m,n,h}^{\operatorname{pl}}\pr{t}}_{t\in[0,1]}\to 
 \pr{ t^{m-d} I_d\pr{h_d}\pr{t}}_{t\in [0,1]},
\end{equation}
where $I_d$ is defined as in \eqref{eq:def_integrale_Id}.
\end{Theorem}

This result improves that in \cite{MR4294337}, where the same conclusion was 
deduced under a more restrictive condition than
\eqref{eq:cond_suffisante_PI_holderien_Ustats}, namely, 
under integrability of $\abs{h\pr{\xi_{\intent{1,m}}}}^{p\pr{\alpha}}
\pr{\log\pr{1+\abs{h\pr{\xi_{\intent{1,m}}}}}}^m$. 

Convergence of the finite-dimensional distributions is guaranteed by
Corollary~1 of \cite{MR740907}. The most challenging part is the proof of 
tightness, which rests on a criterion based on tails of differences of 
$U$-statistics. These ones can be expressed as a weighted $U$-statistic hence 
the tools developed in this paper are appropriated.

\subsection{A moment inequality for incomplete 
$U$-statistics}\label{subsec:moment_inequ_incomp}

The computation of a $U$-statistic of order $m$ requires the computation of
$\binom{n}m$ terms, which is of order $n^m$ and can lead to practical
difficulties. For this reason, Blom introduced in \cite{MR0474582}
the concept of incomplete $U$-statistics. The main idea is to put for each
index $\pr{\gri}$ a random weight taking the value $0$ or $1$.
There are several ways for defining such a random weights:
\begin{itemize}
 \item sampling without replacement: we pick without replacement $N$ $m$-uples
of the form
$\gri=\pr{i_\ell}_{\ell\in\intent{1,m}}$ where $1\leq i_1<\dots<i_m\leq n$.
\item sampling with replacement:  we pick with replacement $N$ $m$-uples
of the form
$\gri=\pr{i_\ell}_{\ell\in\intent{1,m}}$ where $1\leq i_1<\dots<i_m\leq n$.
\item Bernoulli sampling: consider for
each $\gri=\pr{i_\ell}_{\ell\in\intent{1,m}}$ satisfying $1\leq 
i_1<\dots<i_m\leq n$, a random
variable $a_{n;\gri}$ taking the value $1$ with probability $p_n$ and
$0$ with probability $1-p_n$. Moreover, we assume that the family
$\pr{a_{n,\gri}}_{\gri\in\inc^m_n}$ is independent and also
independent of the sequence $\pr{\xi_i}_{i\in\Z}$.
\end{itemize}
The weak convergence of incomplete $U$-statistics 
has been established in \cite{MR0753810}.
Rates in the law of large numbers were obtained in \cite{MR2915089}.
Recent papers include incomplete $U$-statistics in a high dimensional setting 
\cite{MR4025737} and based on a triangular array 
\cite{MR4289844}.

 It turns out
that in the setting of the applications of our inequalities, the most
convenient way of picking the
random weights is Bernoulli sampling. Indeed, after having applied
Theorem~\ref{thm:deviation_inequality_Ustats}, we will have to control moments
of
random variables of the form $\sum a^p_{n;\gr{i}}$,
where the sum runs over the indexes $i_k$ and $k$ belongs to
some subset $K$ of $\intent{1,m}$. In this case, it is possible
to bound this by a sum of random variables following a binomial
distribution, and this will be helpful in the sequel.
We thus define
\begin{equation}\label{eq:def_Ustat_incomplete}
U_{m,n}^{\operatorname{inc}}\pr{h}=\sum_{\gri\in\inc^m_n
}a_{n;\gri}h\pr{\xi_{\gri}},
\end{equation}
where the family $\pr{a_{n;\gri}}_{\gri\in\inc^m_n}$ is i.i.d., independent of
$\pr{\xi_i}_{i\geq 1}$, and $a_{n;\gri}$ takes the
value $1$ (respectively $0$) with probability $p_n$
(respectively $1-p_n$).

\begin{Theorem}\label{thm:moment_inequality_incomplete_Ustat}
  Let $m\geq 2$ be an integer, $\pr{\xi_i}_{i\in\Z}$ be an i.i.d.\
sequence taking values in a  measurable space $\pr{S,\Sca}$, let
$\pr{\B,\norm{\cdot}_{\B}}$ be a separable $r$-smooth Banach space
and let $h\colon S^m\to\B$ be a symmetric function which is
degenerated of order $d$ with respect to $\pr{\xi_i}_{i\in\Z}$. Let
$\pr{a_{n;\gri}}_{\gri
\in\inc^m_n }$ be an i.i.d.\ sequence of Bernoulli random variables with
parameter $p_n$, which is independent of $\pr{\xi_i}_{i\geq 1}$. Let 
$U_{m,n}^{\operatorname{inc}}\pr{h}$ defined as in 
\eqref{eq:def_Ustat_incomplete}. The following inequality takes place:
\begin{equation}\label{eq:moment_inequ_incomp}
\E{\norm{U_n^{\operatorname{inc}}}_{\B}^q}
\leq
C\pr{\B,m,p,q}\pr{n^{q\pr{m-d}+dq/p}p_n^q+n^{mq/p}p_n^{q/p}+n^mp_n
}\E{\norm{h\pr{\xi_{\intent{1,m}}}}_{\B}^q },
\end{equation}
where the constant $C\pr{\B,m,p,q}$ depends only on $\B$, $m$, $p$ and $q$.
\end{Theorem}
Note that the degeneracy degree of $h$ appears only 
in the first term of the right hand side of 
\eqref{eq:moment_inequ_incomp}. When $d=m$, the second term of the right hand 
side of \eqref{eq:moment_inequ_incomp} dominates the first one. If $n^mp_n\leq 
1$, then $n^{mq/p}p_n^{q/p}\leq n^{m}p_n$, since $q\geq p$.

In \cite{MR4503429}, a moment inequality for incomplete 
real-valued $U$-statistics has been established. No degeneracy is 
assumed, and the $U$-statistic is only centered hence this would  correspond to 
the case $d=1$. The authors look at the $\el^1$ norm of the $U$-statistic.
\section{Proofs}

\subsection{A "good-$\lb$-inequality" for conditional moment of
  order $p$ of some Banach-valued martingales}

One of the key ingredients of the proof is a
so-called "good-$\lb$-inequality", that is, an inequality of the form
$g\pr{\beta\lb}\leq K\pr{\delta}g\pr{\lb}+h\pr{\lb}$, where $\beta>1$, 
$\delta,\lb>0$, $K\pr{\delta}$ is small as $\delta$ is close to $0$ and 
$g$ and $h$ are tail functions of some random
variables. Such inequalities are very helpful in order to derive moment
inequalities \cite{MR0365692,MR1331198,MR995572} or
deviation inequalities \cite{MR3077911,MR4046858}. In our
context, we will need a good-$\lb$-inequality for tails of conditional
moments of partial sums of a martingale difference sequence.

\begin{Proposition}\label{prop:good_lambda}
 Let $p\in (1,2]$, $\pr{\B,\norm{\cdot}_{\B}}$ be a separable Banach space 
such that $C_{p,\B}$ defined as in \eqref{eq:definition_constant_Banach_space}  
is finite, let $\pr{\xi_i}_{i\in\Z}$ be an i.i.d.\ sequence and let
$\pr{\xi'_i}_{i\in\Z}$ be an independent copy of $\pr{\xi_i}_{i\in\Z}$. For a
subset $J$ of $\Z$, denote by $\Fca_J$ (respectively $\Fca'_J$)
 the $\sigma$-algebra generated by the random variables
 $\xi_i,  i\in J$ (respectively  $\xi'_i,  i\in J$). Suppose that $\pr{D_i}_{i\geq 1}$ is a martingale difference
 sequence with respect to the filtration
 $\pr{\Fca_{\intent{1,i}  }\vee \Fca'_{\Z} }_{i\geq 0}$ and let
 $S_j=\sum_{i=1}^jD_i$. Let
 $p\in (1,r]$ and assume that $\E{\norm{D_i}_{\B}^p}$ is finite
for each $i$.  Then for each $J_0\subset\Z$, each positive $\lb$,
$\beta>1$
 and $0<\delta<1$, the following inequality holds:
 \begin{multline}
  \PP\pr{\max_{1\leq j\leq n}\pr{\E{\norm{S_j}_{\B}^p \mid
  \Fca_{\Z}\vee \Fca'_{J_0}}}^{1/p}>\beta\lb  }\\
  \leq \pr{\frac{p}{p-1}}^p\frac{C_{p,\B}\delta^p}{\pr{\beta-1-\delta}^p }\PP\pr{\max_{1\leq j\leq n}\pr{\E{\norm{S_j}^p_{\B} \mid
  \Fca_{\Z}\vee \Fca'_{J_0}}}^{1/p}> \lb  }
  \\ +  \PP\pr{\max_{1\leq i\leq n}\pr{\E{\norm{D_i}_{\B}^p \mid
  \Fca_{\Z}\vee \Fca'_{J_0}}}^{1/p}>\delta\lb  }
  +\PP\pr{\pr{\sum_{i=1}^n\E{\norm{D_i}_{\B}^p\mid \Fca_{\intent{1,i-1}}\vee \Fca'_{J_0} } }^{1/p}>\delta \lb},
 \end{multline}
where $C_{p,\B}$ is defined as in \eqref{eq:definition_constant_Banach_space}.
\end{Proposition}
\begin{proof}
 We introduce the random variables
 \begin{equation}
  M^*=\pr{\max_{1\leq j\leq n}\E{\norm{\sum_{i=1}^jD_i}_{\B}^p\mid\Fca_{\Z}}}^{1/p},
 \end{equation}
  \begin{equation}
  \Delta^*=\max_{1\leq i\leq n}\E{\norm{D_i}^p_{\B}\mid\Fca_{\Z}\vee \Fca'_{J_0}},
 \end{equation}
\begin{equation}
 \Gamma_\ell=\pr{\sum_{i=1}^\ell \E{\norm{D_i}_{\B}^p\mid\Fca_{\intent{1,i-1} 
}\vee \Fca'_{J_0} }  }^{1/p},
\end{equation}
\begin{equation}
T\pr{t}=\min\ens{1\leq j\leq n, \pr{\E{\norm{S_j}_{\B}^p \mid
  \Fca_{\Z}\vee \Fca'_{J_0}}}^{1/p}\geq t},t>0
\end{equation}
and 
\begin{equation}
\sigma=\min\ens{1\leq \ell\leq n\mid \max\ens{\pr{\E{\norm{D_\ell}_{\B}^p\mid 
\Fca_{\Z}\vee \Fca'_{J_0}}}^{1/p}, 
\Gamma_{\ell+1}}\geq \delta\lb},
\end{equation}
with the convention that $\inf\emptyset=+\infty$.
We now define 
\begin{equation}
D'_i=\ind{T\pr{\lb}<i\leq \min\ens{T\pr{\beta\lb},\sigma,n}}D_i.
\end{equation}
Notice that by Lemma~\ref{lem:two_cond_exp}, the set $T\pr{\lb}<i\leq \min\ens{T\pr{\beta\lb},\sigma,n}$ 
belongs to $\Fca_{\intent{1,i-1} }\vee\Fca'_{\Z}$. Consequently, the sequence 
$\pr{D'_i}_{i\geq 1}$ is a martingale difference sequence with respect to 
the filtration $\pr{\Fca_{\intent{1,i}}\vee \Fca'_{\Z}}_{i\geq 1}$. Moreover, 
the following equality holds by definition of $T\pr{t}$ and $\sigma$:
\begin{equation}\label{eq:egalite_tps_arret_maximum}
\ens{M^*\geq\beta\lb}\cap\ens{\max\ens{\Gamma_n,\Delta^*} \leq\delta\lb}
=\ens{T\pr{\lb}\leq n}\cap\ens{\sigma=\infty}
\end{equation}
and the following inclusion holds:
\begin{equation}\label{eq:key_inclusion}
\ens{T\pr{\lb}\leq n}\cap\ens{\sigma=\infty}
\subset\ens{\pr{\max_{1\leq j\leq n}\E{\norm{\sum_{i=1}^jD'_i}_{\B}^p\mid\Fca_{\Z}\vee\Fca'_{\Z_0}}}^{1/p} >\pr{\beta-1-\delta}\lb}.
\end{equation}
To see this, notice that 
\begin{equation}
\sum_{i=1}^jD'_i=\sum_{i=T\pr{\lb}+1}^{\min\ens{\sigma,T\pr{\beta\lb},j}}
D_i=\sum_{i=1}^{\min\ens{\sigma,T\pr{\beta\lb},j}}
D_i-\sum_{i=1}^{T\pr{\lb}-1}
D_i-D_{T\pr{\lb}}
\end{equation}
and applying twice the conditional Minkowski's inequality
$\pr{\E{\abs{X+Y}^p\mid\Gca}}^{1/p}\leq \pr{\E{\abs{X}^p\mid\Gca}}^{1/p}+\pr{\E{\abs{Y}^p\mid\Gca}}^{1/p}$ with $\Gca=\Fca_{\Z}\vee \Fca'_{J_0}$ gives 
\begin{multline}\label{eq:minoration_martingale_transformee}
\max_{1\leq j\leq n}\pr{\E{\norm{\sum_{i=1}^jD'_i}_{\B}^p\mid\Fca_{\Z}}}^{1/p}
\geq \max_{1\leq j\leq n}\pr{\E{\norm{\sum_{i=1}^{\min\ens{\sigma,T\pr{\beta\lb},j}}D'_i}_{\B}^p\mid\Fca_{\Z}\vee \Fca'_{J_0}}}^{1/p}\\
-\pr{\E{\norm{\sum_{i=1}^{ T\pr{\lb}-1}D'_i}_{\B}^p\mid\Fca_{\Z}\vee \Fca'_{J_0} 
 }}^{1/p}-\pr{\E{\norm{D_{T\pr{\lb}}}_{\B}\mid\Fca_{\Z}\vee \Fca'_{J_0}  
}}^{1/p}
\end{multline}
and since 
\begin{equation}
 \max_{1\leq j\leq n}\pr{\E{\norm{\sum_{i=1}^{\min\ens{\sigma,T\pr{\beta\lb},j}}D'_i}_{\B}^p\mid\Fca_{\Z}\vee \Fca'_{J_0}}}^{1/p}\ind{\ens{T\pr{\lb}\leq n}\cap\ens{\sigma=\infty}}\geq \beta\lb,
\end{equation}
\begin{equation}
\pr{\E{\norm{\sum_{i=1}^{ T\pr{\lb}-1}D'_i}_{\B}^p\mid\Fca_{\Z}\vee \Fca'_{J_0}}}^{1/p}\ind{\ens{T\pr{\lb}\leq n}\cap\ens{\sigma=\infty}}\leq \lb 
\end{equation}
\begin{equation}
\pr{\E{\norm{D_{T\pr{\lb}}}_{\B}\mid\Fca_{\Z}\vee \Fca'_{J_0}  
}}^{1/p}\ind{\ens{T\pr{\lb}\leq n}\cap\ens{\sigma=\infty}}\leq \delta\lb ,
\end{equation}
we get \eqref{eq:key_inclusion} in view of 
\eqref{eq:minoration_martingale_transformee}.
Combining \eqref{eq:egalite_tps_arret_maximum} with \eqref{eq:key_inclusion} gives 
\begin{multline*}
\PP\pr{\ens{M^*\geq\beta\lb}\cap\ens{\max\ens{\Gamma_n,\Delta^*} \leq\delta\lb}}\leq 
\PP\pr{ \max_{1\leq j\leq n}\E{\norm{\sum_{i=1}^jD'_i}_{\B}^p\mid \Fca_{\Z}\vee \Fca'_{J_0}}>\pr{\beta-1-\delta}^p\lb^p}\\
\leq 
\PP\pr{ \E{\max_{1\leq j\leq n}\norm{\sum_{i=1}^jD'_i}_{\B}^p\mid\Fca_{\Z}\vee \Fca'_{J_0}}>\pr{\beta-1-\delta}^p\lb^p},
\end{multline*}
then Markov's and Doob's inequality give that 
\begin{equation}
\PP\pr{\ens{M^*\geq\beta\lb}\cap\ens{\max\ens{\Gamma_n,\Delta^*} \leq\delta\lb}}
 \leq \pr{\frac{p}{p-1}}^{p} \frac{1}{\pr{\beta-1-\delta}^p\lb^p}
\E{\norm{\sum_{i=1}^nD'_i}_{\B}^p}.
\end{equation}
From \eqref{eq:definition_constant_Banach_space}, we deduce that 
\begin{equation}
\PP\pr{\ens{M^*\geq\beta\lb}\cap\ens{\max\ens{\Gamma_n,\Delta^*} \leq\delta\lb}}
 \leq \pr{\frac{p}{p-1}}^{p} \frac{C_{p,\B}}{\pr{\beta-1-\delta}^p\lb^p}
\sum_{i=1}^n \E{\norm{D'_i}_{\B}^p}.
\end{equation}
Moreover, by definition of $D'_i$ and $\Fca_{\intent{1,i-1}}\vee\Fca'_{J_0}$-measurability of 
$\ens{T\pr{\lb}<i\leq \min\ens{T\pr{\beta\lb},\sigma,n}}$, 
one gets that 
\begin{multline}
\PP\pr{\ens{M^*\geq\beta\lb}\cap\ens{\max\ens{\Gamma_n,\Delta^*} \leq\delta\lb}}\\
 \leq \pr{\frac{p}{p-1}}^{p} \frac{C_{p,\B}}{\pr{\beta-1-\delta}^p\lb^p}
\sum_{i=1}^n \E{\ind{T\pr{\lb}<i\leq \min\ens{T\pr{\beta\lb},\sigma,n}} \E{\norm{D_i}_{\B}^p\mid \Fca_{\intent{1,i-1}}\vee\Fca'_{J_0} } }.
\end{multline}
By definition of $T\pr{\lb}$ and $\sigma$, we infer that 
\begin{multline}
\PP\pr{\ens{M^*\geq\beta\lb}\cap\ens{\max\ens{\Gamma_n,\Delta^*} \leq\delta\lb}}\\
 \leq \pr{\frac{p}{p-1}}^{p} \frac{C_{p,\B}}{\pr{\beta-1-\delta}^p\lb^p}
\sum_{i=1}^n \E{\ind{M^*>\lb } \ind{\Gamma_i\leq\delta\lb}\E{\norm{D_i}_{\B}^p\mid \Fca_{\intent{k,i-1} }\vee\Fca'_{J_0}  } }.
\end{multline}
We then conclude by the elementary inequality 
$\sum_{i=1}^n\ind{\sum_{k=1}^iY_k\leq x}Y_i\leq x$ for non-negative random variables $Y_i$ and a positive $x$. This ends the proof of 
Proposition~\ref{prop:good_lambda}.
\end{proof}

The previous good $\lb$-inequality lead to the following inequality, expressing
the tail of the conditional moment of a Banach-valued martingale in terms of a
functional of the tails of the conditional increments and conditional moment of
order $p$.

\begin{Corollary}\label{cor:good_lb}
 Let $p\in (1,2]$. There exists a function $f_{1,p}\colon \R_{>0}\times  \R_{>0}
  \to\R_{>0}$ such that if  $\pr{\B,\norm{\cdot}_{\B}}$ is a separable Banach space 
such that $C_{p,\B}$ defined as in \eqref{eq:definition_constant_Banach_space}  
is finite, 
each i.i.d.\ sequence $\pr{\xi_i}_{i\in\Z}$ and its independent copy
$\pr{\xi'_i}_{i\in\Z}$, each $\B$-valued martingale difference sequence
$\pr{D_i}_{i\geq 1}$
 with respect to the filtration
 $\pr{\Fca_{\intent{1,i}}\vee\Fca'_{\Z} }_{i\geq 0}$, where 
$\Fca_I=\sigma\pr{\xi_i,i\in I}$, $\Fca'_J=\sigma\pr{\xi'_j,j\in J}$ each 
$J_0\subset \Z$ and each $q,t>0$,
 \begin{multline}\label{eq:cor_good_lb}
  \PP\pr{\max_{1\leq j\leq n}\pr{\E{\norm{S_j}_{\B}^p \mid
  \Fca_{\Z}\vee\Fca'_{J_0}}}^{1/p}>t  }\\
  \leq f_{1,p}\pr{q,C_{p,\B}}\int_0^1\PP\pr{\max_{1\leq i\leq n}\pr{\E{\norm{D_i}_{\B}^p \mid
 \Fca_{\Z}\vee\Fca'_{J_0}}}^{1/p}>tu  }u^{q-1}du\\
+f_{1,p}\pr{q,C_{p,\B}}\int_0^1\PP\pr{\pr{\sum_{i=1}^n\E{\norm{D_i}_{\
B}^p\mid\Fca_{\intent{1,i-1}}\vee\Fca'_{J_0} }}^{1/p}>tu}
  u^{q-1}du,
 \end{multline} 
 where $S_j=\sum_{i=1}^j D_i$.
\end{Corollary}
\begin{proof}
Let us fix the Banach space $\pr{\B,\norm{\cdot}_{\B}}$ into consideration,  
$1<p\leq r$, $q>0$, $\pr{\xi_i}_{i\in\Z}$ as well as $J_0\subset\Z$ and 
$\pr{D_i}_{i\geq 
1}$. We choose 
$\beta_0>1$ and $\delta_0<1$ (depending only on $p$, $q$ and $C_{p,\B}$) such that 
\begin{equation}
\frac{\beta_0^{q}\delta_0^p}{\pr{\beta_0-1-\delta_0}^p}\pr{\frac{p}{p-1}}^pC_{p,\B}\leq 1.
\end{equation} 
Defining 
\begin{equation}
g\pr{t}=  \PP\pr{\max_{1\leq j\leq n}\pr{\E{\norm{S_j}_{\B}^p \mid
 \Fca_{\Z}\vee\Fca'_{J_0}}}^{1/p}>t  }
\end{equation}
 and 
\begin{multline}
h\pr{t}= \PP\pr{\max_{1\leq i\leq n}\pr{\E{\norm{D_i}_{\B}^p \mid
 \Fca_{\Z}\vee\Fca'_{J_0}}}^{1/p}>t  }\\+
  \PP\pr{\pr{\sum_{i=1}^n\E{\norm{D_i}_{\B}^p\mid\Fca_{\intent{1,i-1} }  \vee\Fca'_{J_0} }}^{1/p}>t},
\end{multline}
Proposition~\ref{prop:good_lambda} shows that for each positive $t$, 
\begin{equation}
g\pr{\beta_0t}\leq \beta_0^{-q}g\pr{t}+h\pr{t\delta_0}.
\end{equation}
 Observing that $g$ is non-increasing, we infer that 
 \begin{align}
q\int_0^1u^{q-1}h\pr{tu}du&=\sum_{\ell\geq 0}\int_{\beta_0^{-\ell-1}}^{\beta_0^{-\ell}}
qu^{q-1}h\pr{tu}du\\
&\geq \sum_{\ell\geq 0}h\pr{t\beta_0^{-\ell-1}}\int_{\beta_0^{-\ell-1}}^{\beta_0^{-\ell}}
qu^{q-1}du\\
&\geq \sum_{\ell\geq 0}h\pr{t\beta_0^{-\ell-1}}\pr{\beta_0^{-\ell q}-\beta_0^{-\ell q-q}}\\
&\geq \pr{1-\beta_0^{-q}}\sum_{\ell\geq 0}\beta_0^{-\ell q}
\pr{g\pr{\beta_0\beta_0^{-\ell-1}t/\delta_0    }-\beta_0^{-q}
g\pr{ \beta_0^{-\ell-1}t/\delta_0    } }
 \end{align}
 and defining $c_\ell= g\pr{ \beta_0^{-\ell}t/\delta_0    }$, the previous bound 
 gives (accounting that $\beta_0^{-\pr{\ell+1} q}c_{\ell+1}\leq \beta_0^{-\pr{\ell+1} q}\to 0$)
 \begin{equation}
 q\int_0^1u^{q-1}h\pr{tu}\geq  \pr{1-\beta_0^{-q}}\sum_{\ell\geq 0}
 \pr{\beta_0^{-\ell q}
 c_\ell-\beta_0^{-\pr{\ell+1} q}c_{\ell+1}}
 = c_0
 \end{equation}
 hence 
 \begin{equation}
 g\pr{ t/\delta_0  }  \leq  q\int_0^1u^{q-1}h\pr{tu}du\frac 1{1-\beta_0^{-q}}.
 \end{equation}
 To conclude, we apply this bound with $t$ replaced by $\delta_0t$ and do the substitution $v=\delta_0 u$.
\end{proof}
\subsection{Key ingredient in the proof of
Theorem~\ref{thm:deviation_inequality_Ustats}}

\subsubsection{Statement}
We would like to prove Theorem~\ref{thm:deviation_inequality_Ustats} by induction on the order of the considered $U$-statistic. An application of Corollary~\ref{cor:good_lb} 
would reduce to bound the tail of a $U$-statistic of lower order, but the term 
involving the conditional moment of order $p$ cannot be treated directly with 
the induction assumption.

For this reason, we will consider the following assertion $A\pr{m}$, depending 
on the parameter $m\geq 1$, which is defined as follows:
for each $p\in (1,2]$ and each $k\geq 0$, there exists a function $f_{m,k,p}\colon  \R_{>0}\times 
\R_{>0}\to \R_{\geq 0}$ such that the following holds: if
$\pr{\B,\norm{\cdot}_{\B}}$ is a separable Banach space for which $C_{p,\B}$
defined as in \eqref{eq:definition_constant_Banach_space} is finite,
$\pr{\xi_i}_{i\in\Z}$ is an i.i.d.\ sequence taking values
in a measurable space $\pr{S,\Sca}$ and $\pr{\xi'_i}_{i\in\Z}$ is an independent 
copy of $\pr{\xi_i}_{i\in\Z}$, $h_{\gri}\colon S^m\times S^k\to\B$, 
$\gri=\pr{i_\ell}_{\ell\in\intent{1,m}}\in\inc_n^m$,
are measurable functions such that for each $\ell_0\in \intent{1,m+k}$ ,
\begin{equation}\label{eq:degeneree_etape_rec}
\E{h_{\gri}\pr{\xi_{\intent{1,m+k}} }
\mid   \xi_{\intent{1,m+k}\setminus\ens{\ell_0}  } }=0,
\end{equation}
and $J_0$ is a subset of $\intent{m+1,m+k+1}$, 
the following inequality takes place: 
\begin{multline}\label{eq:key_step_deviation_inequality_Ustats}
\PP\pr{\max_{m\leq n\leq N}\pr{
\E{
\norm{\sum_{\gri\in\inc_n^m}h_{\gri}\pr{\xi_{\gri},\xi'_{\intent
{ m + 1,m+
k+1} } }}_{\B}^p\mid\Fca_{\Z}\vee \Fca'_{J_0} }}^{1/p}>t }\\
\leq f_{m,k,p}\pr{q,C_{p,\B}} 
\sum_{  J\subset  \intent{1,m}}\sum_{\gr{i_J}\in\N^J }\int_0^1 u^{q-1}
g_{\gr{i_J}}\pr{ut}du,
\end{multline} 
where we write 
\begin{equation}
   \xi'_{\intent{a,b}}=\pr{\xi'_a,\dots,\xi'_b},
\end{equation}
use the small abuse of notation 
$$
h_{\gri}\pr{\xi_{\gri},\xi'_{\intent{m+1,m+k+1}} }
=h_{\gri}\pr{\xi_{i_1},\dots,\xi_{i_m},\xi'_{m+1},\dots,\xi'_{m+k+1}}
$$
and define
\begin{equation}
g_{\gr{i_J}}\pr{t}=\PP\pr{\pr{ \sum_{\gr{i_{J^c}}: \gr{i_J}+\gr{i_{J^c}}\in 
\inc_n^m } \E{ \norm{ 
h_{\gr{i_J}+\gr{i_{J^c}}}\pr{\xi'_{\intent{1, m+k}  
}}}_{\B}^p  
\mid \Fca'_{J \cup J_0}  }   }^{1/p}>t},
\end{equation}
  $\Fca_I=\sigma\pr{\xi_i,i\in I}$,  $\Fca'_J=\sigma\pr{\xi'_j,j\in J}$ and the convention 
$\Fca_{\emptyset}=\ens{\emptyset,\Omega}$. 
Note that Theorem~\ref{thm:deviation_inequality_Ustats} corresponds to
$A\pr{m}$ with $k=0$ and $J_0= \emptyset$.

In order to perform the induction step, we need to define recursively the functions $f_{m,k,p}$. We define $f_{1,k,p}=f_{1,p}$ for each $k\geq 0$, where $f_{1,p}$ is defined as in Corollary~\ref{cor:good_lb} and 
\begin{equation}\label{eq:def_recursive_}
f_{m+1,k,p}\pr{q,K}=
f_{1,k,p}\pr{q,K} f_{m,k+1,p}\pr{q,K},q,K\in\R_{>0}.
\end{equation}

\subsubsection{The case $m=1$}

We  will show  $A\pr{ 1}$. The term associated with $J=\emptyset$ in \eqref{eq:key_step_deviation_inequality_Ustats} and $J=\ens{1}$ 
correspond to the first and second term of the right hand side of \eqref{eq:cor_good_lb} 
hence it suffices to apply Corollary~\ref{cor:good_lb}.

\subsubsection{The case $m=2$}

Although the case $m=2$ is not required from a purely formal point
of view, it will help the understanding of the induction step.
Let $p\in (1,2]$, $k\geq 0$, let $\pr{\B,\norm{\cdot}_{\B}}$ be a
separable Banach space for which the constant $C_{p,\B}$ defined as
in \eqref{eq:definition_constant_Banach_space} is finite, an
i.i.d.\ sequence $\pr{\xi_i}_{i\in\Z}$, an independent copy
$\pr{\xi'_i}_{i\in\Z}$, functions $h_{i,j}\colon S^2\times S^k\to\B$
such that for each $\ell_0\in\intent{1,k+2}$,
\begin{equation}\label{eq:deg_demo_m=2}
  \E{h_{i,j}\pr{\xi_{\intent{1,k+2}}}\mid
     \xi_{\intent{1,k+2}\setminus\ens{\ell_0}}
}=0.
\end{equation}
We have to show that for each $J_0\subset \intent{3,k+2}$ and each
$q,t>0$,
\begin{multline}\label{eq:inegalite_a_montrer_m=2}
  \PP\pr{\max_{2\leq n\leq N}\pr{ \E{
        \norm{\sum_{1\leq i<j\leq
            n}h_{i,j}\pr{\xi_i,\xi_j,\xi'_{\intent{3, k+2}}}
}_{\B}^p\mid\Fca_{\Z}\vee\Fca'_{J_0}
      }  }^{1/p}>t}\\
  \leq f_{2,k,p}\pr{q,C_{p,\B}}\sum_{1\leq i<j\leq N}
  \int_0^1u^{q-1}
  \PP\pr{\pr{\E{ \norm{h_{i,j}\pr{\xi'_1,\xi'_2,\xi'_{\intent{3, k+2}}}}^p_{\B} 
\mid \Fca'_{\ens{1,2 } \cup J_0} }}^{1/p}>tu}
  du\\
  +f_{2,k,p}\pr{q,C_{p,\B}}\sum_{i=1}^{N-1}
  \int_0^1u^{q-1}
  \PP\pr{\pr{\sum_{j=i+1}^N\E{
\norm{h_{i,j}\pr{\xi'_1,\xi'_2,\xi'_{\intent{3, k+2}}}}^p_{\B} \mid 
\Fca'_{\ens{1 } \cup J_0} }}^{1/p}>tu}
  du\\
   +f_{2,k,p}\pr{q,C_{p,\B}}\sum_{j=2}^{N}
  \int_0^1u^{q-1}
  \PP\pr{\pr{\sum_{i=1}^{j-1}\E{
\norm{h_{i,j}\pr{\xi'_1,\xi'_2,\xi'_{\intent{3, k+2}}}        }^p_{\B} \mid 
\Fca'_{\ens{2 } \cup J_0} }}^{1/p}>tu}
  du\\
   +f_{2,k,p}\pr{q,C_{p,\B}}
  \int_0^1u^{q-1}
  \PP\pr{\pr{\sum_{i=1}^{N-1}\sum_{j=i+1}^N\E{
\norm{h_{i,j}\pr{\xi'_1,\xi'_2,\xi'_{\intent{3, k+2}}}
}^p_{\B} \mid  \Fca'_{J_0}} }^{1/p}>tu}
  du.
\end{multline}
To do so, define
\begin{equation}
  D_j=\sum_{i=1}^{j-1}h_{i,j}\pr{\xi_i,\xi_j,\xi'_{\intent{3, k+2}}}.
\end{equation}
In this way, $\sum_{1\leq i<j\leq
n}h_{i,j}\pr{\xi_i,\xi_j,\xi'_{\intent{3, k+2}}}=\sum_{j=2}^nD_j$
and by \eqref{eq:deg_demo_m=2} the sequence $\pr{D_j}_{j\geq 2}$ is a
martingale difference sequence with respect to the filtration
$\pr{\Fca_{\intent{1,j}}\vee\Fca'_{\Z}  }_{j\geq 2}$.
Applying Corollary~\ref{cor:good_lb} gives
\begin{equation}
  \PP\pr{\max_{2\leq n\leq N}\pr{ \E{
        \norm{\sum_{1\leq i<j\leq
            n}h_{i,j}\pr{\xi_i,\xi_j,\xi'_{\intent{3, k+2}}}
}_{\B}^p\mid\Fca_{\Z}\vee\Fca'_{J_0}
      }  }^{1/p}>t}\\
  \leq  P_1+P_2,
\end{equation}
where
\begin{equation*}
 P_1:= f_{1,k,p}\pr{q,C_{p,\B}}\int_0^1u^{q-1}
  \PP\pr{ \max_{2\leq j\leq N}\pr{ \E{
\norm{\sum_{i=1}^{j-1}h_{i,j}\pr{\xi_i,\xi_j,\xi'_{\intent{3, k+2}}  }}
_{\B}^p \mid \Fca_{\Z}\vee\Fca'_{J_0} }  }^{1/p}>tu} du;
\end{equation*}
\begin{equation*}
 P_2:=f_{1,k,p}\pr{q,C_{p,\B}}\int_0^1u^{q-1}
  \PP\pr{ \pr{\sum_{2\leq j\leq N} \E{
\norm{\sum_{i=1}^{j-1}h_{i,j}\pr{\xi_i,\xi_j,\xi'_{\intent{3, k+2}} }}
_{\B}^p \mid \Fca_{\intent{1,j-1}}\vee\Fca'_{J_0} }  }^{1/p}>tu} du.
\end{equation*}
In order to bound $P_1$, we use a union bound and find that
\begin{equation}
  P_1\leq \sum_{j=2}^N
  f_{1,k,p}\pr{q,C_{p,\B}}\int_0^1u^{q-1}
  \PP\pr{  \pr{ \E{
\norm{\sum_{i=1}^{j-1}h_{i,j}\pr{\xi_i,\xi_j,\xi'_{\intent{3, k+2}}   }}
_{\B}^p \mid \Fca_{\Z}\vee\Fca'_{J_0} }  }^{1/p}>tu} du.
\end{equation}
Moreover, since $D_j$ is $\Fca_{\intent{1,j}}\vee
  \Fca'_{\Z}$-measurable, Lemma~\ref{lem:two_cond_exp} gives that
\begin{equation*}
  P_1\leq \sum_{j=2}^N
  f_{1,k,p}\pr{q,C_{p,\B}}\int_0^1u^{q-1}
  \PP\pr{  \pr{ \E{
\norm{\sum_{i=1}^{j-1}h_{i,j}\pr{\xi_i,\xi_j,\xi'_{\intent{3, k+2}}}}
_{\B}^p \mid \Fca_{\intent{1,j}}\vee\Fca'_{J_0} }  }^{1/p}>tu} du.
\end{equation*}
Define for $j\in\intent{2,N}$ the random variable
\begin{equation}\label{eq:def_de_Dj'}
D'_j:=\sum_{i=1}^{j-1}h_{i,j}\pr{\xi_i,\xi'_2,
\xi'_{\intent{3, k+2}}}
.
\end{equation}
Notice that
\begin{equation}
 \E{
\norm{\sum_{i=1}^{j-1}h_{i,j}\pr{\xi_i,\xi_j,\xi'_{\intent{3, k+2}}}}
_{\B}^p \mid \Fca_{\intent{1,j}}\vee\Fca'_{J_0} }
=H\pr{\xi_{\intent{1,j}},\pr{\xi'_k}_{k\in J_0}},
\end{equation}
where
\begin{equation}
  H\pr{x_{\intent{1,d}},\pr{x'_k}_{k\in J_0}}
  =\E{
    \norm{\sum_{i=1}^{j-1}h_{i,j}\pr{x_i,x_j,V\pr{\pr{Y_k}_{k\in
            \intent{3,k+2}}}}}
_{\B}^p   },
\end{equation}
where in vector $V\pr{\pr{Y_k}_{k\in
            \intent{3,k+2}}}$,
        $Y_k=x'_k$ if $k\in J_0$
        and $Y_k=\xi'_k$ otherwise.
Since
\begin{equation}
  \E{
    \norm{\sum_{i=1}^{j-1}h_{i,j}\pr{\xi_i,\xi'_2,V\pr{\pr{Y_k}_{k\in
            \intent{3,k+2}}}}}
_{\B}^p   }=H\pr{\xi_{\intent{1,j-1},\xi'_2,V\pr{\pr{Y_k}_{k\in
            \intent{3,k+2}}} }},
\end{equation}
we also have
that 
\begin{equation}\label{eq:egalites_esp_cond}
 \E{\norm{D_j}^p_{\B}\mid \Fca_{\intent{1,j}}  \vee \Fca'_{J_0}
}=\E{\norm{D'_j}^p_{\B}\mid \Fca_{\intent{1,j-1}}  \vee \Fca'_{J_0}
},
\end{equation}
 which, in terms of $P_1$, translates as
\begin{multline}
  P_1\\\leq \sum_{j=2}^N
  f_{1,k,p}\pr{q,C_{p,\B}}\int_0^1u^{q-1}
  \PP\pr{  \pr{ \E{
\norm{\sum_{i=1}^{j-1}h_{i,j}\pr{\xi_i,\xi'_2,\xi'_{\intent{3, k+2}}
}}
_{\B}^p \mid \Fca_{\intent{1,j-1}}\vee\Fca'_{\ens{2}\cup J_0} }
}^{1/p}>tu} du.
\end{multline}
We then use Corollary~\ref{cor:good_lb} in the following setting:
\begin{itemize}
 \item
$\widetilde{D_i}=\sum_{i=1}^{j-1}h_{i,j}\pr{\xi_i,\xi'_{\intent{2,k + 2}}
}$,
\item $\widetilde{q}=q$,
\item $\widetilde{k}=k+1$,
\item $\widetilde{J_0}=\ens{2}\cup J_0$.
\end{itemize}
After having bounded the term with the maximum by a union bound, we
get
\begin{multline}\label{eq:bound_P_1_case_m=2}
  P_1\leq \sum_{j=2}^N\sum_{i=1}^{j-1}
  f_{1,k,p}\pr{q,C_{p,\B}}f_{1,k+1,p}\pr{q,C_{p,\B}}
  \\
\int_0^1u^{q-1}\int_0^1v^{q-1}
  \PP\pr{  \pr{ \E{
\norm{h_{i,j}\pr{\xi_i,\xi'_2,\xi'_{\intent{3, k+2}}
}}
_{\B}^p \mid \Fca_{\intent{1,j-1}}\vee\Fca'_{\ens{2}\cup J_0} }
}^{1/p}>tu} dudv\\
+\sum_{j=2}^N
f_{1,k,p}\pr{q,C_{p,\B}}f_{1,k+1,p}\pr{q,C_{p,\B}}\\
\int_0^1u^{q-1}\int_0^1v^{q}
  \PP\pr{  \pr{\sum_{i=1}^{j-1} \E{
\norm{h_{i,j}\pr{\xi_i,\xi'_2,\xi'_{\intent{3, k+2}}
}}
_{\B}^p \mid \Fca_{\intent{1,i-1}}\vee\Fca'_{\ens{2}\cup J_0} }
}^{1/p}>tuv} dudv.
\end{multline}
By Lemma~\ref{lem:removing_ind_cond_exp}, \eqref{eq:def_recursive_}, the 
elementary inequality 
\begin{equation}\label{eq:bound_int_Puv}
 \int_0^1\int_0^1 u^{q-1}v^{q-1}\PP\pr{Y>tuv}dudv
 \leq \int_0^1 u^{q-1}\PP\pr{Y>tu}du,
\end{equation}
and \eqref{eq:bound_P_1_case_m=2}, we deduce that 
\begin{multline*}
 P_1\leq f_{2,k,p}\pr{q,C_{p,\B}}\sum_{j=2}^N\sum_{i=1}^{j-1}  \int_0^1 
 u^{q-1} 
  \PP\pr{  \pr{ \E{
\norm{h_{i,j}\pr{\xi_i,\xi'_2,\xi'_{\intent{3, k+2}}
}}
_{\B}^p \mid \Fca_{\ens{i}}\vee\Fca'_{\ens{2}\cup J_0} }
}^{1/p}>tu} du\\
+\sum_{j=2}^N
f_{2,k,p}\pr{q,C_{p,\B}} 
\int_0^1u^{q-1} 
  \PP\pr{  \pr{\sum_{i=1}^{j-1} \E{
\norm{h_{i,j}\pr{\xi_i,\xi'_2,\xi'_{\intent{3, k+2}}
}}
_{\B}^p \mid \Fca'_{\ens{2}\cup J_0} }
}^{1/p}>tu} du.
\end{multline*}
Since the vectors 
$\pr{\xi_1,\xi'_{\intent{2, k+2}}}$ and 
$\pr{\xi'_1,\xi'_{\intent{2, k+2}}}$ 
have the same distribution, the right hand side corresponds to the sum of the 
first and third terms in \eqref{eq:inegalite_a_montrer_m=2}. 

Let us now bound $P_2$. By \eqref{eq:egalites_esp_cond}, we can write 
$P_2/f_{1,k,p}\pr{q,C_{p,\B}}$ as 
\begin{equation}
 \int_0^1u^{q-1}
  \PP\pr{ \pr{\sum_{2\leq j\leq N} \E{
\norm{\sum_{i=1}^{j-1}h_{i,j}\pr{\xi_i,\xi'_2,\xi'_{\intent{3, k+2}}}}
_{\B}^p \mid \Fca_{\intent{1,j-1}}\vee\Fca'_{J_0} }  }^{1/p}>tu} du.
\end{equation}
In order to control this quantity, we introduce the Banach space 
$\til{\B}=\ens{\pr{x_j}_{j\in\intent{2,N}}}$ endowed with the norm 
$\norm{\pr{x_j}_{j\in\intent{2,N}}}_{\til{\B}}
=\pr{\sum_{j=2}^N\norm{x_j}_{\B}^p}^{1/p}$. We use Corollary~\ref{cor:good_lb} 
in the following setting:
\begin{itemize}
 \item $\til{D_i}=\pr{ 
 h_{i,j}\pr{\xi_i,\xi'_{\intent{2,k+2} }  
}}_{j\in\intent{2,N}}$,
 \item $\til{q}=q$, 
 \item $\til{k}=k+1$,
 \item $\til{J_0}=J_0$.
\end{itemize}
Using \eqref{eq:bound_int_Puv} and \eqref{eq:def_recursive_}, we get that 
\begin{multline}
 P_2\leq f_{2,k,p}\pr{q,C_{p,\B}}
 \int_0^1 u^{q-1}
 \PP\pr{\max_{1\leq i\leq N} \pr{  \E{
 \norm{\til{D_{i}}}_{\til{\B}}^p  \mid\Fca_{\Z}\vee \Fca'_{J_0}} }^{1/p}}du\\
 +f_{2,k,p}\pr{q,C_{p,\B}}
 \int_0^1 u^{q-1}
 \PP\pr{  \pr{ \sum_{i=1}^{N} \E{
 \norm{\til{D_{i}}}_{\til{\B}}^p  \mid\Fca_{\intent{1,i-1}}\vee \Fca'_{J_0}} 
}^{1/p}}du\\
\leq f_{2,k,p}\pr{q,C_{p,\B}}
 \sum_{i=1}^N \int_0^1 u^{q-1}
 \PP\pr{ \pr{  \E{
 \sum_{j=i+1}^N\norm{h_{i,j}\pr{\xi_i,\xi'_2,\xi'_{\intent{3, k+2}}}}_{\B}^p  
\mid\Fca_{\Z}\vee \Fca'_{J_0}} }^{1/p}}du\\
+ f_{2,k,p}\pr{q,C_{p,\B}}\int_0^1 u^{q-1} 
\PP\pr{  \pr{ \sum_{i=1}^{N}\sum_{j=i+1}^N \E{
 \norm{ h_{i,j}\pr{\xi_i,\xi'_2,\xi'_{\intent{3, k+2}}}}_{ \B}^p  
\mid\Fca_{\intent{1,i-1}}\vee \Fca'_{J_0}} 
}^{1/p}}
du.
\end{multline}
To conclude, we notice that 
$$\E{\norm{h_{i,j}\pr{\xi_i,\xi'_2,\xi'_{\intent{3, k+2}}}}_{\B}^p  
\mid\Fca_{\Z}\vee \Fca'_{J_0}}
=\E{\norm{h_{i,j}\pr{\xi'_1,\xi'_2,\xi'_{\intent{3, k+2}}}}_{\B}^p  
\mid \Fca'_{J_0}}$$
and that 
$$
\E{
 \norm{ h_{i,j}\pr{\xi_i,\xi'_2,\xi'_{\intent{3, k+2}}}}_{ \B}^p  
\mid\Fca_{\intent{1,i-1}}\vee \Fca'_{J_0}} 
=\E{
 \norm{ h_{i,j}\pr{\xi'_1,\xi'_2,\xi'_{\intent{3, k+2}}}}_{ \B}^p  
\mid \Fca'_{\ens{1}\cup J_0}},
$$
and get the second and fourth terms of the right hand side of 
\eqref{eq:inegalite_a_montrer_m=2}.
\subsubsection{Induction step: from $m$ to $m+1$}
Suppose that $A\pr{m}$ is true and let us show that 
$A\pr{m+1}$ is true.

Let $p\in (1,2]$ be fixed.
Let $k\geq 0$ be fixed, as well as a separable Banach space
$\pr{\B,\norm{\cdot}_{\B}}$ for which $C_{p,\B}$ defined as in
\eqref{eq:definition_constant_Banach_space} is finite, an i.i.d.\
sequence $\pr{\xi_i}_{i\in\Z}$ and an independent copy $\pr{\xi'_i}_{i\in\Z}$, 
functions $h_{\gri,i_{m+1}}\colon S^{m+1}\times S^k\to \B$ such that
for each $\ell\in\intent{1,m+k+1}$,  and 
$\gri\in\inc^m$ and $i_{m+1}>i_m$, 
  \begin{equation}\label{eq:degeneree_etape_rec2}
\E{h_{\gri,i_{m+1}}\pr{
\xi_{\intent{1, m+k+1}}}
\mid   \xi_{\intent{1,m+k+1}\setminus\ens{\ell_0}  } 
}=0.
\end{equation}
Finally, let $q,t>0$, $J_0\subset \intent{m+2,m+2+k}$ and $N\geq m+1$
be fixed. We have to show that
\begin{multline}\label{eq:key_step_deviation_inequality_Ustats_m_plus_1}
\PP\pr{\max_{m+1\leq n\leq N}\pr{
\E{
\norm{\sum_{\pr{\gri,i_{m+1}}\in\inc_n^{m+1}}h_{\pr{\gri,i_{m+1}}}\pr{\xi_{
\pr{\gri,i_{m+1}}},\xi'_{\intent{m+2,m+k+1} }} }_{\B}^p\mid 
\Fca_{\Z}\vee \Fca'_{J_0} }}^{1/p}>t }\\
\leq f_{m+1,k,p}\pr{q,C_{p,\B}} 
\sum_{  J\subset  \intent{1,m+1}}\sum_{\gr{i_J}\in\N^J }\int_0^1 u^{q-1}
g_{\gr{i_J}}\pr{ut}du,
\end{multline}
where 
\begin{equation}
g_{\gr{i_J}}\pr{t}=\PP\pr{\pr{ \sum_{\gr{i_{J^c}}: \gr{i_J}+\gr{i_{J^c}}\in 
\inc_N^{m+1} } \E{ 
\norm{ 
h_{\gr{i_J}+\gr{i_{J^c}}}\pr{ \xi'_{\intent{1,m+k+1} }}   
}_{\B}^p  
\mid \Fca'_{ J \cup  J_0} }   }^{1/p}>tu}.
\end{equation}
 It will turn out that after an application of Corollary~\ref{eq:cor_good_lb} to an appropriate martingale difference sequence, 
the contribution of the maximum of increments will correspond to the subsets of
$ \intent{1,m+1}$ containing $m+1$ while the contribution of the sum of
conditional moments will correspond to the subsets of $ \intent{1,m+1}$ which 
do not contain $m+1$. To do so, let us define for $j\geq m+1$ the random 
variable 
\begin{equation}
 D_j:=\sum_{\gri \in\inc_{j-1}^m 
}h_{\pr{\gri,j}}\pr{\xi_{\gri},\xi_j,\xi'_{\intent{m+2,m+k+1} }}.
\end{equation}
In this way, 
\begin{equation}
 \sum_{\pr{\gri,i_{m+1}}\in\inc_n^{m+1}}h_{\pr{\gri,i_{m+1}} }\pr{\xi_{
\pr{\gri,i_{m+1}}} ,\xi'_{\intent{m+2,m+k+1} }}=\sum_{j=m+1}^n D_j
\end{equation}
and $\pr{D_j}_{j\geq m+1}$ is a martingale difference sequence with respect to  
the filtration $\pr{\Fca_{\intent{1,j}}\vee\Fca'_{\Z} }_{j\geq m}$. Consequently, an application of 
Corollary~\ref{cor:good_lb} gives 
\begin{multline}\label{eq:key_step_deviation_inequality_Ustats_m_plus_1_etape_1}
\PP\pr{\max_{m+1\leq n\leq N}\pr{
\E{
\norm{\sum_{\pr{\gri,i_{m+1}}\in\inc_n^{m+1} 
}h_{\pr{\gri,i_{m+1}}}\pr{\xi_{\pr{\gri,i_{m+1}}},\xi'_{\intent{m+2,m+k+1} }} 
}_{\B}^p\mid 
\Fca_{\Z}\vee \Fca'_{J_0} }}^{1/p}>t }\\
\leq f_{1,k,p}\pr{q,C_{p,\B}}\int_0^1\PP\pr{\max_{m+1\leq j\leq 
N}\pr{\E{\norm{D_j}_{\B}^p \mid
  \Fca_{\Z}\vee \Fca'_{J_0}}}^{1/p}>tu  }u^{q-1}du\\
+f_{1,k,p}\pr{q,C_{p,\B}}\int_0^1\PP\pr{\pr{\sum_{j=m+1}^N\E{\norm{D_j}_{\B}^p\mid
\Fca_{\intent{1,j-1}}\vee \Fca'_{J_0} }}^{1/p}>tu}u^{q-1}du=:P_1+P_2.
\end{multline}
Let us estimate $P_1$. A union bound gives 
\begin{equation}
 P_1\leq 
f_{1,k,p}\pr{q,C_{p,\B}}\sum_{j=m+1}^{N}\int_0^1\PP\pr{\pr{\E{\norm{D_j}_{\B}^p \mid
  \Fca_{\Z}\vee \Fca'_{J_0}}}^{1/p}>tu  }u^{q-1}du.
\end{equation}
Using the fact that $D_j$ is $\Fca_{\intent{1,j}}\vee\Fca'_{\Z}$-measurable, one gets 
in view of Lemma~\ref{lem:two_cond_exp} that 
\begin{equation}\label{eq:bound_of_A_induction}
 P_1\leq 
f_{1,k,p}\pr{q,C_{p,\B}}\sum_{j=m+1}^{N}\int_0^1\PP\pr{\pr{\E{\norm{D_j}_{\B}^p \mid
  \Fca_{\intent{1,j}}\vee \Fca'_{J_0}}}^{1/p}>tu  }u^{q-1}du.
\end{equation}
Moreover,  defining 
\begin{equation}
D'_j=\sum_{\gri \in\inc_{j-1}^m 
}h_{\gri,j}\pr{\xi_{\gri},\xi'_{\intent{m+1,m+k+1}}}
\end{equation}
one derive by the same arguments as those who led to 
\eqref{eq:egalites_esp_cond} that
\begin{equation}\label{eq:same_law_cond_Dj_Dprime_j}
\E{\norm{D_j}_{\B}^p \mid
  \Fca_{\intent{1,j}}\vee \Fca'_{J_0}}\overset{\mathrm{law}}{=}
\E{\norm{D'_j}_{\B}^p \mid
  \Fca_{\intent{1,j-1}}\vee \Fca'_{J_0\cup\ens{m+1}}}
\end{equation}
and  \eqref{eq:bound_of_A_induction} allows us to infer that 
\begin{equation}\label{eq:bound_of_A_induction_2}
 P_1\leq 
f_{1,k,p}\pr{q,C_{p,\B}}\sum_{j=m+1}^{N}\int_0^1\PP\pr{\pr{\E{\norm{D'_j}_{\B}^p \mid
  \Fca_{\intent{1,j-1}}\vee \Fca'_{J_0\cup\ens{m+1}}}}^{1/p}>tu
}u^{q-1}du.
\end{equation} 
Since $D'_j$ is $\Fca_{\intent{1,j-1}}\vee\Fca'_{\Z}$-measurable, Lemma~\ref{lem:two_cond_exp} gives that 
\begin{equation}
\E{\norm{D'_j}_{\B}^p \mid
  \Fca_{\intent{1,j-1}}\vee \Fca'_{J_0\cup\ens{m+1}}}
  =\E{\norm{D'_j}_{\B}^p \mid
  \Fca_{\Z}\vee \Fca'_{J_0\cup\ens{m+1}}}
\end{equation}
hence \eqref{eq:bound_of_A_induction_2} can be rephrased as 
\begin{equation}\label{eq:bound_of_A_induction_3}
 P_1\leq 
f_{1,k,p}\pr{q,C_{p,\B}}\sum_{j=m+1}^{N}\int_0^1\PP\pr{\pr{\E{\norm{D'_j}_{\B}^p \mid
  \Fca_{\Z}\vee \Fca'_{J_0\cup\ens{m+1}}}}^{1/p}>tu  }u^{q-1}du.
\end{equation} 

Now we are in position to use the induction assumption in the following context: for a fixed $j\in\intent{m+1,N}$,
\begin{itemize} 
\item $\til{k}=k+1$,
\item $\til{h_{\gri}}\colon S^m\times S^{k+1}\to \B$ defined by 
\begin{equation}
\til{h_{\gri}}\pr{x_1,\dots,x_m,x_{m+1},\dots,x_{m+k+1}}
=h_{\gri,j}\pr{x_1,\dots,x_m,x_{m+1},\dots,x_{m+k+1}},
\end{equation}
\item $\til{q}=q+1$,
\item $\til{J_0}=J_0\cup\ens{m+1}$,
\item $\til{N}=j-1$, 
\end{itemize}
which gives, in view of \eqref{eq:def_recursive_},
\begin{equation}\label{eq:bound_of_A_induction_4}
 P_1\leq 
f_{m+1,k,p}\pr{q,C_{p,\B}}   
\sum_{J\subset\intent{1,m}}\sum_{\gr{i_J}\in\N^J}\sum_{j=m+1}^{N}\int_0^1\int_0^
1g ^ { \pr{j}}_{\gr{i_J}}\pr{tuv}u^{q-1}v^{q }dudv,
\end{equation} 
where 
\begin{equation}
g^{\pr{j}}_{\gr{i_J}}\pr{t}=
\PP\pr{\pr{ \sum_{\gr{i_{J^c}}: \gr{i_J}+\gr{i_{J^c}}\in \inc_{j-1}^m } \E{ 
\norm{ 
h_{\gr{i_J}+\gr{i_{J^c}}+j\gr{e_{m+1}}}\pr{\xi'_{\intent{1, 
m + k }  }} }_{\B}^p  
\mid \Fca'_{J \cup J_0}  }   }^{1/p}>t}.
\end{equation}
Defining for $J\subset \intent{1,m+1}$ such that $m+1  \in J$, 
\begin{equation}
g_{\gr{i_J}}\pr{t}= \PP\pr{\pr{ \sum_{\gr{i_{J^c}}: 
\gr{i_{J\cup\ens{m+1}}}+\gr{i_{J^c}}\in \inc_{N}^{m+1 }} \E{ \norm{ 
h_{\gr{i_J}+\gr{i_{J^c}}+i_{m+1}\gr{e_{m+1}}}\pr{\xi'_{\intent{1, 
m + k }  }  } }_{\B}^p  
\mid \Fca'_{J \cup J_0}  }   }^{1/p}>t}
\end{equation}
and using the elementary bound \eqref{eq:bound_int_Puv}
 gives 
\begin{equation}\label{eq:bound_P1}
 P_1\leq 
f_{m+1,k,p}\pr{q,C_{p,\B}}  \sum_{\substack{J\subset\intent{1,m+1}\\ m+1\in J } 
}\sum_{\gr{i_J}\in\N^J} \int_0^1 g^{\pr{j}}_{\gr{i_J}}\pr{tu}u^{q-1} du.
\end{equation}
Let us now bound $P_2$ defined by \eqref{eq:key_step_deviation_inequality_Ustats_m_plus_1_etape_1}.

 One can derive
\begin{equation}\label{eq:same_cond_Dj_Dprime_j_bis}
\E{\norm{D_j}_{\B}^p \mid
  \Fca_{\intent{1,j-1}}\vee \Fca'_{J_0}}=
\E{\norm{D'_j}_{\B}^p \mid
  \Fca_{\intent{1,j-1}}\vee \Fca'_{J_0}}
\end{equation}
using that $\E{g\pr{U,V_1}\mid W}=\E{g\pr{U,V_2}\mid W}$ if 
$V_1$ and $V_2$ have the same distribution and are both independent of 
$\sigma\pr{U,W}$ and $D_j$ and $D'_j$ only differ by the replacement 
of $\xi_j$ in $D_j$ by $\xi'_m$. Moreover, using Lemma~\ref{lem:two_cond_exp}, 
the following equality holds $\E{\norm{D'_j}_{\B}^p \mid
  \Fca_{\intent{1,j-1}}\vee \Fca'_{J_0}}=\E{\norm{D'_j}_{\B}^p \mid
  \Fca_{\Z}\vee \Fca'_{J_0}}$. Consequently, 
  \begin{equation}
  P_2\leq 
  f_{1,k,p}\pr{q,C_{p,\B}}\int_0^1\PP\pr{\pr{\sum_{j=m+1}^N\E{\norm{D'_j}_{\B}^p\mid
\Fca_{\Z}\vee \Fca'_{J_0} }}^{1/p}>tu}u^{q-1}du.
  \end{equation}
In order to use the induction assumption, we need to view the term 
$\sum_{j=m+1}^N\E{\norm{D'_j}_{\B}^p\mid
\Fca_{\Z}\vee \Fca'_{J_0} }$ as the conditional expectation of the $p$-th power norm of an element of a new Banach space. This leads us to introduce the 
Banach space $\pr{\til{\B},\norm{\cdot}_{\til{\B}}}$ by 
\begin{equation}
\til{\B}=\ens{\pr{x_j}_{j\in\intent{m+1,N}},x_j\in\B}, \quad 
\norm{\pr{x_j}_{j\in\intent{m+1,N}}}_{\til{{\B}}}=\pr{\sum_{j=m+1}^N\norm{x_j}_{
\B}^p  }^{1/p}.
\end{equation}
We apply the induction assumption to the following setting 
\begin{itemize} 
\item $\til{k}=k+1$,
\item $\til{h_{\gri}}\colon S^m\times S^{k+1}\to \til{\B}$ defined by 
\begin{equation}
\til{h_{\gri}}\pr{x_1,\dots,x_m,x_{m+1},\dots,x_{m+k+1}}
=\pr{h_{\gri,j}\pr{x_1,\dots,x_m,x_{m+1},\dots,x_{m+k+1}}\ind{i_m\leq j-1}  
}_{j\in\intent{m+1,N}},
\end{equation}
\item $\til{q}=q+1$,
\item $\til{J_0}=J_0$,
\item $\til{N}=N$.
\end{itemize}
Using \eqref{eq:def_recursive_} and \eqref{eq:bound_int_Puv}, we find that 
\begin{equation}\label{eq:bound_P2}
P_2\leq f_{m+1,k,p}\pr{q,C_{p,\B}}\int_0^1 
\sum_{J\subset\intent{1,m}}\sum_{\gr{i_J}\in\N^J}g_{\gr{i_J}}
\pr{tw}w^{q-1}dw,
\end{equation}
where 
\begin{equation}
g_{\gr{i_J}}\pr{t}=\PP\pr{\pr{ \sum_{\gr{i_{J^c}}: 
\gr{i_J}+\gr{i_{J^c}}\in \inc_N^m } 
\E{ \norm{ 
\til{h_{\gr{i_J}+\gr{i_{J^c}}  }}\pr{\xi'_{\intent{1,m+k}}  } 
}_{\til{\B}}^p  
\mid \Fca'_{J \cup J_0}  }   }^{1/p}>t}.
\end{equation}
Going back to the $\norm{\cdot}_{\B}$-norm and the expression of 
$\til{h_{\gri}}$ finally gives 
\begin{equation}
P_2\leq f_{m+1,k,p}\pr{q,C_{p,\B}}\int_0^1 
\sum_{\substack{J\subset\intent{1,m+1} \\ m+1\notin J} 
}\sum_{\gr{i_J}\in\N^J}G_{\gr{i_J}}\pr{tw}
w^{q-1}dw
\end{equation}
with 
\begin{equation*}
G_{\gr{i_J}}\pr{t}=\PP\pr{\pr{ \sum_{\gr{i_{J^c\cup\ens{m+1}}}: 
\gr{i_J}+\gr{i_{J^c\cup\ens{m+1}}}\in \inc_N^{m+1} } \E{ \norm{ 
 h_{\gr{i_J}+\gr{i_{J^c\cup\ens{m+1}}}} \pr{\xi'_{\intent{1,m+k}} } }_{ \B}^p  
\mid \Fca'_{J \cup J_0}  }   }^{1/p}>t}.
\end{equation*}
The combination of \eqref{eq:bound_P1} with \eqref{eq:bound_P2} concludes 
the proof that $A\pr{m}$ is true for each $m$ and that of 
Theorem~\ref{thm:deviation_inequality_Ustats}.

\subsection{Proof of Corollary~\ref{cor:ineg_deviation_deg_ordre_d_sym}}
We start from \eqref{eq:decomposition_somme_U_stat_deg}, which gives 
\begin{equation}
\max_{m\leq n\leq N}\norm{U_{m,n}\pr{h}}_{\B}
\leq K_m \sum_{c=d}^m \max_{k\leq n\leq N} 
\norm{U_{c,n}\pr{h^{\pr{c}}}}_{\B}N^{m-c},
\end{equation}
where $h^{\pr{c}}$ is defined as in \eqref{eq:def_hc} and $K_m$ 
depends only on $m$. As a consequence, one has 
\begin{equation}
\PP\pr{\max_{m\leq n\leq N}\norm{U_{m,n}\pr{h}}_{\B}>t}
\leq \sum_{c=d}^m \PP\pr{\max_{k\leq n\leq N} 
\norm{U_{c,n}\pr{h^{\pr{c}}}}_{\B}N^{m-c}> t/\pr{mK_m}}.
\end{equation}
 By Corollary~\ref{cor:meme_noyau} applied with $m$ replaced by $c$ 
 and $h$ by $h^{\pr{c}}$, we derive that 
 \begin{multline}
 \PP\pr{\max_{m\leq n\leq N}\norm{U_{m,n}\pr{h}}_{\B}>t}
\leq  
K\pr{m,p,q,\B} \sum_{c=d}^mN^c 
\int_0^1u^{q-1}\PP\pr{\norm{h^{\pr{c}} \pr{
\xi_{\intent{1,c}}} }_{\B} >t\frac{u}{mK_m}} du\\
+K\pr{m,p,q,\B}\sum_{c=d}^m
\sum_{j=1}^c N^{j}\int_0^1 u^{q-1}
\PP\pr{N^{\frac{c-j}{p}}\pr{   \E{ \norm{ h^{\pr{c}}\pr{\xi_{\intent{1,c}}}   
}_{\B}^p  \mid \xi_{ \intent{1,j}}  }   }^{1/p}>t\frac{u}{mK_m}}du\\
+K\pr{m,p,q,\B}\pr{mK_m}^qt^{-q}\sum_{c=d}^mN^{cq/p}\pr{ 
\E{\norm{h^{\pr{c}} \pr{\xi_{\intent{1,c}}}}^p_{\B}}  }^{q/p}.
 \end{multline}
We conclude using the fact that 
$$
\E{ \norm{ h^{\pr{c}}\pr{\xi_{\intent{1,c}}}   
}_{\B}^p  \mid \xi_{ \intent{1,j}}  }\leq \kappa_m
\E{ \norm{ h \pr{\xi_{\intent{1,m}}}   
}_{\B}^p  \mid \xi_{ \intent{1,j}}  }.
$$
\subsection{Proof of the results of Subsection~\ref{subsec:LGN_comp_Ustats}}

\begin{proof}[Proof of Theorem~\ref{thm:fonction_maximale_Ustat}]

It suffices to show that 
\begin{equation}\label{eq:convergence_series_pour_fct_max}
  \sum_{N\geq 2}
  \PP\pr{\max_{m\leq n\leq 2^N}\norm{\sum_{\gri
  \in\inc^m_n}h\pr{\xi_{\gri}  } }_{\B}>
   2^{1+Nm/p }} \leq \E{\norm{h\pr{\xi_{\intent{1,m}}}}_{\B}^p},
\end{equation}
then \eqref{eq:fonction_maximale_Ustat} can be deduced by 
replacing $h$ by $h/t$ and noticing that 
\begin{align}
  \sup_{n\geq m}\frac{1}{n^{m/p }}\norm{\sum_{\gri
  \in\inc^m_n}h\pr{\xi_{\gri} } }_{\B}
&  \leq \sup_{N\geq 1}
  \frac{1}{2^{m/pN }}\max_{m\leq n\leq 2^{N+1}}\norm{\sum_{\gri
  \in\inc^m_n}h\pr{\xi_{\gri}  } }_{\B}\\
&\leq 2^{m/p}\sup_{N\geq 2}
  \frac{1}{2^{m/pN }}\max_{m\leq n\leq 2^{N }}\norm{\sum_{\gri
  \in\inc^m_n}h\pr{\xi_{\gri}  } }_{\B}.
\end{align}

In order to show \eqref{eq:convergence_series_pour_fct_max}, define for a fixed 
$N\geq 2$
\begin{equation}
 p_N:=\PP\pr{\max_{m\leq n\leq 2^N}\norm{\sum_{\gri
  \in\inc^m_n}h\pr{\xi_{\gri}  } }_{\B}>
   2^{1+Nm/p}}
\end{equation}
and define the function $h_{\leq}\colon S^m\to\B$  and 
$h_{>}\colon S^m\to\B$ by 
\begin{equation}
 h_{\leq}\pr{x_1,\dots,x_m}=h\pr{x_1,\dots,x_m}\ind{
 \norm{h\pr{x_1,\dots,x_m}}_{\B}\leq 2^{Nm/p}} \mbox{ and }
\end{equation}
\begin{equation}
 h_{>}\pr{x_1,\dots,x_m}=h\pr{x_1,\dots,x_m}\ind{
 \norm{h\pr{x_1,\dots,x_m}}_{\B}>2^{Nm/p}}.
\end{equation}
Since \eqref{eq:deg_fct_max} does not hold with 
$h$ replaced by $h_\leq$, we define 
\begin{equation}
 \til{h_\leq}\pr{x_1,\dots,x_m}=
 \sum_{I\subset\intent{1,m}}\pr{-1}^{m-\abs{I}}\E{h_{\leq}   
\pr{x_I,\xi_{\intent{1,m}\setminus I}}}
\end{equation}
where $\pr{x_I,\xi_{\intent{1,m}\setminus I}}$ is the 
element of $S^m$ whose coordinate $i$ is $\xi_i$ if 
$i\in I$ and $\xi_i$ otherwise. We define similarly 
\begin{equation}
 \til{h_>}\pr{x_1,\dots,x_m}=
 \sum_{I\subset\intent{1,m}}\pr{-1}^{m-\abs{I}}\E{h_{>}   
\pr{x_I,\xi_{\intent{1,m}\setminus I}}}.
\end{equation}
By construction and assumption \eqref{eq:deg_fct_max}, 
$h=\til{h_\leq}+\til{h_>}$, which shows that
\begin{equation}
p_N\leq p_{N,\leq}+p_{N,>}, \mbox{ with } p_{N,\leq}= \PP\pr{\max_{m\leq n\leq 
2^N}\norm{\sum_{\gri
  \in\inc^m_n}\til{h_{\leq}}\pr{\xi_{\gri}  } }_{\B}>
   2^{Nm/p}},
\end{equation}
\begin{equation}
 p_{N,>}= \PP\pr{\max_{m\leq n\leq 2^N}\norm{\sum_{\gri
  \in\inc^m_n}\til{h_{>}}\pr{\xi_{\gri}  } }_{\B}>
   2^{Nm/p}}.
\end{equation}
Let us bound $p_{N,\leq}$. Using Markov's inequality 
and noticing that 
 \begin{equation}\label{eq:deg_fct_max_noy_tronque}
  \forall\ell_0\in\intent{1,m},\quad 
\E{\til{h_{\leq}}\pr{\xi_{\intent{1,m}}}\mid  \xi_{
  \intent{1,m}\setminus\ens{\ell_0}
  } }=0, 
 \end{equation}
we find by a use of \eqref{eq:moment_ineq_degenerated_s=p_meme_noyau} with $p$ 
replaced by $q$ that
 \begin{align*}
p_{N,\leq }&\leq 2^{-Nmr/p}
\E{\max_{m\leq n\leq 
2^N}\norm{\sum_{\pr{i_\ell }
  \in\inc^m_n}\til{h_{\leq}}\pr{\xi_{\intent{1,m}} } }_{\B}^r}\\
  &\leq K\pr{m,p,,r\B} 2^{-Nmr/p}2^{Nm}
\E{\norm{\til{h_{\leq}}\pr{\xi_{\intent{1,m}}}}_{\B}^r }. 
 \end{align*}
Then noticing that 
$\til{h_{\leq }}\pr{\xi_{\intent{1,m}}}
=\sum_{I\subset\intent{1,m}}\pr{-1}^{m-\abs{I}}\E{h_{\leq}   
\pr{\xi_{\intent{1,m}}}\mid  \xi_I }$
we get that 
\begin{equation}
 p_{N,\leq}\leq  K\pr{m,p,r,\B} 2^{Nm\pr{1-r/p}}
 \E{\norm{ h_{\leq}\pr{\xi_{\intent{1,m}}}}_{\B}^r }.
\end{equation}
Denoting by $H$ the random variable $\norm{h\pr{\xi_{\intent{1,m}}}}_{\B}$, we 
obtained that 
\begin{equation}
 p_{N,\leq}\leq K\pr{m,p,r,\B} 2^{Nm\pr{1-r/p}}
 \E{H^r\ind{H\leq 2^{mN/p}} }
\end{equation}
and from the elementary bound 
$\sum_{N\geq 2}2^{Nm\pr{1-r/p}}\ind{H\leq 2^{mN/p}}\leq 
\kappa_{p,q,m}H^{p-r}$, we infer that 
$ \sum_{N\geq 2} p_{N,\leq}<\infty$.

It remains to show the convergence of $ \sum_{N\geq 2} p_{N,>}$. To do so, we 
use Markov's inequality combined with the observations that 
$\til{h_{> }}\pr{\xi_{\intent{1,m}}}
=\sum_{I\subset\intent{1,m}}\pr{-1}^{m-\abs{I}}\E{h_{>}   
\pr{\xi_{\intent{1,m}} }\mid  \xi_I }$
and 
$$\max_{m\leq n\leq 2^N}\norm{\sum_{\gri
  \in\inc^m_n}\til{h_{>}}\pr{\xi_{\gri} } }_{\B}\leq  
\sum_{\gri
  \in\inc^m_{2^N}}
\norm{\til{h_{>}}\pr{\xi_{\gri}  } }_{\B},$$
we infer that 
\begin{equation}
 p_{N,>}\leq 2^{-Nm/p}2^{mN}\E{H\ind{H>2^{mN/p}}}.
\end{equation}
and the elementary bound $\sum_{N\geq 2}2^{mN\pr{1-1/p}}
\ind{H>2^{mN/p}}\leq c_p  H^{p-1}$ allows to conclude. This 
ends the proof of Theorem~\ref{thm:fonction_maximale_Ustat}.
\end{proof}

\begin{proof}[Proof of Theorem~\ref{thm:sup_Ustats}]
Suppose now that $\alpha>0$. By Hoeffding's decomposition
\eqref{eq:decomposition_somme_U_stat_deg}, it suffices to prove that
for each positive $\eps$,
\begin{equation}
  \sum_{N=0}^\infty 2^{N\pr{\gamma+1}}
  \PP\pr{\sup_{n\geq 2^N}n^\alpha
    \frac 1{\binom
nc}\norm{U_{c,n}\pr{h^{\pr{c}}}}_{\B}>\eps}<\infty.
\end{equation}
Writing
\begin{align*}
 \PP\pr{\sup_{n\geq 2^N}n^\alpha
    \frac 1{\binom
      nc}\norm{U_{c,n}\pr{h^{\pr{c}}}}_{\B}>\eps}
  &=\PP\pr{\bigcup_{k=1}^\infty\ens{\max_{2^{N+k-1}\leq n\leq
        2^{N+k}}n^\alpha
    \frac 1{\binom
      nc}\norm{U_{c,n}\pr{h^{\pr{c}}}}_{\B}>\eps}}\\
&\leq \sum_{k=1}^{\infty}\PP\pr{\max_{2^{N+k-1}\leq
n\leq 2^{N+k}}n^\alpha
    \frac 1{\binom
      nc}\norm{U_{c,n}\pr{h^{\pr{c}}}}_{\B}>\eps}\\
  &\leq
  \sum_{k=1}^{\infty}\PP\pr{\max_{2^{N+k-1}\leq
n\leq 2^{N+k}}
\norm{U_{c,n}\pr{h^{\pr{c}}}}_{\B}>K\eps 2^{\pr{N+k}\pr{c-\alpha} } },
\end{align*}
where $K$ depends only on $c$ and $\alpha$, we are reduced to prove
that for each positive $\eps$,
\begin{equation}
  \sum_{N=1}^\infty\sum_{k=1}^\infty2^{N\gamma}
  \PP\pr{\max_{c\leq
n\leq 2^{N+k}}
\norm{U_{c,n}\pr{h^{\pr{c}}}}_{\B}> \eps 2^{\pr{N+k}\pr{c-\alpha} }
}<\infty.
\end{equation}
To do so, we use Corollary~\ref{cor:meme_noyau} in the case with
$m$ replaced by $c$, $h$ by $h^{\pr{c}}$, $N$ by $2^{N+k}$, $q$ that
will be specified later and $t=\eps 2^{\pr{N+k}\pr{c-\alpha} }$ (note
that symmetry of $h^{\pr{c}}$ implies that the summand in the second
term of the right hand side of
\eqref{eq:deviation_inequality_Ustats_meme_noyau} depends only on the
cardinal of the set $J$). We are thus reduced to show that for each
$c\in\intent{d,m}$ and $j\in\intent{0,c}$,
\begin{equation}
 \sum_{N=1}^\infty\sum_{k=1}^\infty
 a_{N,k,c}<\infty,
\end{equation}
where
\begin{equation*}
  a_{N,k,c}:=
2^{N\pr{\gamma+1}}2^{\pr{N+k}j}\int_0^1\PP\pr{2^{\frac{N+k}{r}\pr{c-j}
}
   \pr{\E{\norm{h^{\pr{c}}\pr{\xi_{\intent{1,c}} }
}_{\B}^r\mid\xi_{\intent{1,j}} } }^{1/r}>\eps 2^{\pr{N+k}\pr{c-\alpha}
}u }u^{q-1}du.
\end{equation*}
Doing the change of index $\ell=N+k$ for a fixed $N$,
switching the sums and using the fact that $\sum_{N=1}^\ell
2^{N\pr{\gamma+1}}\leq c_\gamma 2^{\ell\pr{\gamma+1}}$,
we are reduced to prove that
for each
$c\in\intent{d,m}$ and $j\in\intent{0,c}$,
\begin{equation}\label{eq:LGN_control_max}
\sum_{\ell=0}^\infty
2^{\ell\pr{\gamma+1+j}}\int_0^1\PP\pr{2^{\frac{\ell}r \pr { c - j }}
   \pr{\E{\norm{h^{\pr{c}}\pr{\xi_{\intent{1,c}}}
}_{\B}^r\mid\xi_{\intent{1,j}} } }^{1/r}>\eps 2^{\ell\pr{c-\alpha}}u
}u^{q-1}du<\infty.
\end{equation}
Using $\sum_{\ell=0}^\infty 2^{\ell \beta }
\PP\pr{Y>2^{\ell \beta'}}\leq C_{\beta,\beta'}
\E{Y^{\beta/\beta'}}$, the series involved in
\eqref{eq:LGN_control_max} is convergent as long as
\begin{equation}\label{eq:Hjr_dans_Lgamma}
 \pr{\E{\norm{h^{\pr{c}}\pr{\xi_{\intent{1,c}}  }
}_{\B}^r\mid\xi_{\intent{1,j}}  } }^{1/r}
\in\el^{q\pr{\gamma,c,j}},
\end{equation}
where
\begin{equation}
 q\pr{\gamma,c,j}
 =\frac{\gamma+j+1}{c-\alpha+\frac{j-c}r}.
\end{equation}
Since 
\begin{equation}
 \E{\norm{h^{\pr{c}}\pr{\xi_{\intent{1,c}} }
}_{\B}^r\mid\xi_{\intent{1,j}}}\leq K_c 
\E{\norm{h\pr{\xi_{\intent{1,m}}}}_{\B}^r\mid\xi_{\intent{1,j}} }=H_{j,r}^r,
\end{equation}
\eqref{eq:Hjr_dans_Lgamma} will be satisfied as only as 
$H_{j,r}\in \el^{q\pr{\gamma,c,j}}$ for each $c$ such that $c\geq 
\max\ens{d,j}$. We conclude by noticing that $q\pr{\gamma,c,j}$ is decreasing 
in $c$.
\end{proof}
\subsection{Proof of Theorem~\ref{thm:WIP_Holder_Ustats}}
As mentioned right after the statement of Theorem~\ref{thm:WIP_Holder_Ustats}, 
the convergence of the finite-dimensional distributions is already contained in 
Corollary~1 of \cite{MR740907}. Therefore, it suffices to prove tightness of 
$\pr{n^{d/2-m} \Uca_{m,n,h}^{\operatorname{pl}}}_{n\geq m}$ in $\Hca_\alpha^o$.

Using Hoeffding's decomposition (cf. 
\eqref{eq:decomposition_somme_U_stat_deg}), one can decompose the process 
$\pr{\Uca_{m,n}\pr{h,t}}_{t\in [0,1]}$ as a sum of similar processes 
associated with the kernels $h^{\pr{c}}$, $c\in\intent{d+1,m}$. It suffices to 
prove that each of them are tight. Since required normalization for the 
original process, namely $n^{m-d/2}$, is bigger than $n^{m-c/2}$, it suffices 
to prove tightness of $\pr{n^{-d/2} 
\Uca_{d,n,h^{\pr{d}}}^{\operatorname{pl}}}_{n\geq m}$ in $\Hca_\alpha^o$. 
Using Proposition~1.1 in \cite{MR4294337} with 
$X_k=\sum_{\gri\in\inc^{d-1}_k  }
h^{\pr{d}}\pr{\xi_{\gri},\xi_k  }$ and denoting 
$S_N=\sum^N_{k=1}X_k$, we are reduced to prove that for each positive $\eps$, 
\begin{equation}\label{eq:tension_holder}
 \lim_{J\to\infty}\limsup_{n\to\infty}
 \sum_{j=J}^{\ent{\log_2n}}\sum_{k=0}^{2^j-1}\PP\pr{
 \abs{S_{\ent{n\pr{k+1}2^{-j}}   }-S_{\ent{nk2^{-j}} }  } >n^{d/2}2^{-
 \alpha j}\eps }=0.
\end{equation}
To do so, we apply Theorem~\ref{thm:deviation_inequality_Ustats} 
with $m$ replaced by $d$, $t=n^{d/2}2^{-\alpha j}\eps $ (with fixed $n$, $j$ 
and $k$), $\B=\R$, $p=2$, $q=p\pr{\alpha}+1$ and 
for $\gri=\pr{i_\ell}_{\ell\in\intent{1,d}}$, 
\begin{equation} 
h_{\gri}\pr{x_{\gri}}=
\begin{cases}
h^{\pr{d}}\pr{x_{\gri}}&
\mbox{if }i_d\in \intent{\ent{nk2^{-j}},\ent{n\pr{k+1}2^{-j}}},\\
 0&\mbox{otherwise.}
\end{cases}
\end{equation}
We define 
$Y_0:=\pr{\E{\pr{h^{\pr{d}}\pr{\xi_{\intent{1,d}}}}^2} 
}^{1/2}$ and the random variable
$$
Y:=\max_{1\leq a\leq d}\pr{
\E{\pr{h^{\pr{d}}\pr{ \xi_{\intent{1,d}}}}^2 \mid 
\xi_{\intent{1,a}}   }}^{1/2},
$$
 which come into play in the right hand side of 
\eqref{eq:deviation_inequality_Ustats}. Note that the assumption 
\eqref{eq:cond_suffisante_PI_holderien_Ustats}
 implies 
that 
\begin{equation}\label{eq:tail_Y}
 \lim_{t\to\infty}t^{p\pr{\alpha}}\PP\pr{Y>t  
 }=0.
\end{equation}
Moreover, the sum over $J$ 
in \eqref{eq:deviation_inequality_Ustats} will be split according to the
case where 
$d$ belongs to the set $J$ or not. If $d\in J$, the sum over $\gr{i_J}$ can we 
written as $\sum_{i_d\in \intent{\ent{nk2^{-j}},\ent{n\pr{k+1}2^{-j}}}}
\sum_{\gr{i_{J\setminus \ens{d}}}}$ and the number of summed terms 
does not exceed $2n2^{ -j}\pr{nk2^{1-j}}^{\card{J}-1}$ while the sum 
over $\gr{i_{J^c}}$ contains at most $\pr{nk2^{1-j}}^{d-\card{J}}$ terms.

When $d\notin J$ and $J$ is not empty, the sum over $\gr{i_J}$ contains at 
most $\pr{nk2^{1-j}}^{\card{J}}$ elements and that over $\gr{i_{J^c}}$
at most $ 2n2^{-j}\pr{nk2^{1-j}}^{d-1-\card{J}}  $ elements. All 
these considerations lead to the bound 
\begin{multline}
 P_{n,k,j}:=\PP\pr{
 \abs{S_{\ent{n\pr{k+1}2^{-j}}   }-S_{\ent{nk2^{-j}} }  } >n^{d/2}2^{-
 \alpha j}\eps }\\
 \leq C_\alpha\pr{n2^{-j}}\pr{nk2^{-j}}^{d-1}\int_0^1
 u^{p\pr{\alpha}}\PP\pr{Y> n^{d/2}2^{-\alpha j}\eps u}du\\
  +
 C_\alpha \sum_{a=1}^d n2^{-j}\pr{nk2^{ -j}}^{a-1}
 \int_0^1\PP\pr{   \pr{nk2^{-j}}^{\pr{d-a}/2} Y>C'_\alpha n^{d/2}2^{-
 \alpha j}\eps u  }u^{p\pr{\alpha}}du\\
 +C_\alpha \sum_{a=1}^d \pr{nk2^{-j}}^{a}
 \int_0^1\PP\pr{  \pr{n2^{-j}}^{1/2} \pr{nk2^{-j}}^{\pr{d-a-1}/2} 
Y>C'_\alpha n^{d/2}2^{-
 \alpha j}\eps u  }u^{p\pr{\alpha}}du\\
 +C_\alpha \pr{n^{d/2}2^{-\alpha j}\eps}^{-p\pr{\alpha}-1}
 \pr{
 \pr{n2^{-j}}\pr{nk2^{-j}}^{d-1} Y_0^2   }^{\pr{p\pr{\alpha}+1}/2},
\end{multline}
which can be simplified as follows 
\begin{multline}\label{eq:borne_demo_Holder1}
 P_{n,k,j} 
 \leq 
 C_\alpha n^{d}2^{-jd}k^{d-1}\int_0^1
 u^{p\pr{\alpha}}\PP\pr{Y> n^{d/2}2^{-\alpha j}\eps u}du\\
 +C_\alpha\sum_{a=1}^d   n^ak^{a-1}2^{-ja} 
 \int_0^1\PP\pr{     Y>C'_\alpha 
n^{a/2}k^{\pr{a-d}/2}2^{-j\pr{\alpha-\pr{d-a}/2
 } }\eps u  }u^{p\pr{\alpha}}du\\
 +C_\alpha \sum_{a=1}^d \pr{nk2^{ -j}}^{a}
 \int_0^1\PP\pr{  
Y>C'_\alpha n^{a/2}2^{-
 \pr{\alpha -\pr{d-a}/2   } j}\eps  k^{-\pr{d-a-1}/2}  u  
}u^{p\pr{\alpha}}du\\
+C_\alpha  2^{j\pr{p\pr{\alpha}+\alpha-d\frac{p\pr{\alpha}+1}2 }}  
k^{\pr{d-1}\frac{p\pr{\alpha}+1}2} Y_0^{p\pr{\alpha}}.
\end{multline}
Define 
\begin{equation}
 \tau\pr{R}:=\sup_{t\geq R}t^{p\pr{\alpha}}\PP\pr{Y>t}.
\end{equation}
Note that assumption \eqref{eq:cond_suffisante_PI_holderien_Ustats} 
combined with Lemma~1.4 in \cite{MR3426520} guarantees that 
$\lim_{R\to\infty}\tau\pr{R}=0$. By looking at monotonicity in each variable 
$a,j$ and $k$, we deduce that there exists a constant $K_\alpha$, depending 
only on $\alpha$, such that 
\begin{equation}
 \min_{a\in\intent{0,d}}\:\:\min_{j\in\intent{J,\ent{\log_2n}}}
\:\: \min_{k\in\intent{0,2^j-1}}C'_\alpha 
n^{a/2}k^{\pr{a-d}/2}2^{-j\pr{\alpha-\pr{d-a}/2
 } }\eps
 \geq K_\alpha n^{1/p\pr{\alpha}},
\end{equation}
\begin{equation}
 \min_{a\in\intent{0,d}}\:\:\min_{j\in\intent{J,\ent{\log_2n}}}
 \:\:\min_{k\in\intent{0,2^j-1}}C'_\alpha n^{a/2}2^{-
 \pr{\alpha -\pr{d-a}/2   } j}\eps  k  ^{ \pr{d-a-1}/2} 
 \geq K_\alpha n^{1/p\pr{\alpha}}.
\end{equation}
Consequently, we derive that 
\begin{multline}
 \PP\pr{     Y >C'_\alpha 
n^{a/2}k^{\pr{a-d}/2}2^{-j\pr{\alpha-\pr{d-a}/2
 } }\eps u  }\\ 
 \leq \pr{C'_\alpha 
n^{a/2}k^{\pr{a-d}/2}2^{-j\pr{\alpha-\pr{d-a}/2
 } }\eps u}^{-p\pr{\alpha}}
 \sup_{t>C'_\alpha 
n^{a/2}k^{\pr{a-d}/2}2^{-j\pr{\alpha-\pr{d-a}/2
 } }\eps u}t^{p\alpha}\PP\pr{Y>t}
 \\
 \leq \pr{C'_\alpha 
n^{a/2}k^{\pr{a-d}/2}2^{-j\pr{\alpha-\pr{d-a}/2
 } }\eps u}^{-p\pr{\alpha}}\tau\pr{K_\alpha n^{1/p\pr{\alpha}}u }
\end{multline}
and with a similar reasoning, we infer that
\begin{multline}\label{eq:bound_third_term}
 \PP\pr{  
Y>C'_\alpha n^{a/2}2^{-
 \pr{\alpha -\pr{d-a}/2   } j}\eps  k^{ \pr{d-a-1}/2}  u  
} \\  \leq 
\pr{C'_\alpha n^{a/2}2^{-
 \pr{\alpha -\pr{d-a}/2   } j}k^{ -
\pr{d-a-1}/2}}^{-p\pr{\alpha}}\tau\pr{K_\alpha 
n^{1/p\pr{\alpha}}u }
\end{multline}
and plugging these bounds into \eqref{eq:borne_demo_Holder1} gives 
\begin{multline}\label{eq:borne_demo_Holder2}
 P_{n,k,j} 
 \leq 
K_{1,\alpha} n^{d\pr{1-p\pr{\alpha}/2}} k^{d-1}2^{j\pr{\alpha p\pr{\alpha}-d } 
}  \int_0^1\tau\pr{K_\alpha n^{1/p\pr{\alpha}}u  }  du \\
+ K_{1,\alpha}\sum_{a=1}^d
n^{a\pr{1-p\pr{\alpha}/2}}k^{a-1+p\pr{\alpha}\pr{d-a}/2}  
2^{j\pr{-a+\alpha p\pr{\alpha}-\pr{d-a}p\pr{\alpha}/2}}
 \int_0^1\tau\pr{K_\alpha n^{1/p\pr{\alpha}}u  }  du\\
 +K_{1,\alpha}2^{j\pr{p\pr{\alpha}+\alpha-d\frac{p\pr{\alpha}+1}2 }}  
k^{\pr{d-1}\frac{p\pr{\alpha}+1}2} Y_0^{p\pr{\alpha}}
\end{multline}
where $K_{1,\alpha}$ depends only on $\alpha$ (note that since 
$p\pr{\alpha}>2$, the bound obtained via \eqref{eq:bound_third_term} 
for  the third term of the right hand side of 
\eqref{eq:borne_demo_Holder1} is smaller than the one for the second term). 
Summing over $k$  
furnishes the bound 
\begin{multline}
\sum_{k=0}^{2^j-1} P_{n,k,j}
  \leq 
K_{2,\alpha} n^{d\pr{1-p\pr{\alpha}/2}}2^{j\alpha p\pr{\alpha}} +
  K_{2,\alpha}\sum_{a=1}^d
  n^{a\pr{1-p\pr{\alpha}/2}}2^{j\alpha p\pr{\alpha}}
   \int_0^1\tau\pr{K_\alpha n^{1/p\pr{\alpha}}u  }  du\\
   +K_{2,\alpha}  2^{j\pr{\alpha-1/2}}Y_0^{p\pr{\alpha}},
\end{multline}
then summuing over $j$ gives 
\begin{equation}
   \sum_{j=J}^{\ent{\log_2n}}\sum_{k=0}^{2^j-1}P_{n,k,j}
   \leq K_{3,\alpha} \int_0^1\tau\pr{K_\alpha n^{1/p\pr{\alpha}}u  }  du
   + K_{3,\alpha} 2^{-J \pr{1/2-\alpha}},
\end{equation}
from which \eqref{eq:tension_holder} follows. This ends the proof of 
Theorem~\ref{thm:WIP_Holder_Ustats}.

\subsection{Proof of Theorem~\ref{thm:moment_inequality_incomplete_Ustat}}
By definition of the functions $h^{\pr{c}}$ given in \eqref{eq:def_hc} 
and symmetry of $h$, the incomplete $U$-statistic defined in 
\eqref{eq:def_Ustat_incomplete} can be decomposed as 
\begin{equation}
 U_{m,n}^{\operatorname{inc}}\pr{h}=
 \sum_{c=d}^m \sum_{\gri\in\inc^c_n    } 
a_{n,c;\gri}h^{\pr{c}}\pr{\xi_{\gri}},
\end{equation}
where 
\begin{equation}
 a_{n,c;\gri}=\sum_{\substack{\grj  
\in\inc^m_n  \\ \ens{i_1,\dots,i_c}\subset\ens{j_1,\dots,j_m} 
}}a_{n;\grj}.
\end{equation}
The previous sum can be split according to the position
of the indexes $i_1,\dots,i_c$. We define for
$\gr{k}=\pr{k_u}_{u=1}^c$ such that $1\leq
k_1<\dots<k_c\leq m$ the random variable
\begin{equation}
  a_{n,c,\gr{k},\gri}
  =\sum_{\grj}
  a_{n;\grj},
\end{equation}
where the sum carries over the
$\grj=\pr{j_k}_{k\in\intent{1,m}}\in\inc^m_n$ such that
$j_{k_u}=i_u$ for each $u\in\intent{1,c}$. As a consequence, it
suffices to prove that for each $c\in\intent{1,m}$,
and each $\gr{k}=\pr{k_u}_{u=1}^c$ such that $1\leq
k_1<\dots<k_c\leq m$,
\begin{equation}\label{eq:etape_cle_inegalite_U_stats_incompletes}
  \E{  \norm{\sum_{\gri\in\inc^c_n}   
a_{n,c,\gr{k};\gri}h^{\pr{c}}\pr{\xi_{\gri}}
}_{\B}^q  } 
\leq
C\pr{\B,m,p,q}\pr{n^{q\pr{m-d}+dq/p}p_n^q+n^{mq/p}p_n^{q/p}+n^mp_n }.
\end{equation}
To do so, we first condition on the random variables
$a_{n;\grj}$. Writing
\begin{equation}\label{eq:expression_somme_poids}
 a_{n,c,\gr{k};\gri}=
f_{c,\gr{k};\gri}
\pr{ \pr{a_{n;\grj}}_{\grj
\in\inc^m_n }    }
\end{equation}
and denoting by $\mu_0$ the law of
$a_{n;\grj}$, we have
\begin{equation}
   \E{  \norm{\sum_{ \gri\in\inc^c_n    }
a_{n,c,\gr{k};\gri}h^{\pr{c}}\pr{\xi_{\gri} }
}_{\B}^q  }
\\
=\int
   \E{  \norm{\sum_{
   \gri\in\inc^c_n    }
f_{c,\gr{k};\gri}
\pr{ \pr{x_{n;\grj}}_{\grj
\in\inc^m_n }    }h^{\pr{c}}\pr{\xi_{\gri} }
}_{\B}^q}d\mu_0\pr{x_{n;\grj}},
\end{equation}
that is, the last integral is taken over the $C^m_n$ variables
$x_{n;\grj}$,
$\grj\in\inc^m_n$.
For fixed $x_{n;\grj}\in\ens{0,1}$,
$\grj\in\inc^m_n$,
we apply Corollary~\ref{cor:moment_inequality}
in the following setting: $m=c$ and 
$h_{\gri}\pr{\xi_{\gri}}
=f_{c,\gr{k};\gri}
\pr{ \pr{x_{n;\grj}}_{\grj
\in\inc^m_n }    }h^{\pr{c}}\pr{\xi_{\gri}}$. We get, after having bounded the 
terms of the
form $\E{\pr{   \E{ \norm{
h \pr{\xi_{\intent{1,c}}}   }_{\B}^p  \mid
\xi_J  }   }^{q/p} }$ by
$\E{\norm{h^{\pr{c}}\pr{\xi_{\gri}}}_{\B}^q}$ which is
in tern smaller than a constant depending only on $q/p $ times
$\E{\norm{h\pr{\xi_{\intent{1,m}}}}_{\B}^q}=:H$,
that
\begin{multline}
 \E{  \norm{\sum_{\gri\in\inc^c_n    }
f_{c,\gr{k};\gri }
\pr{ \pr{x_{n;\grj}}_{\grj
\in\inc^m_n }    }h^{\pr{c}}\pr{\xi_{\gri} }
}_{\B}^q}\\
\leq
K\pr{m,p,q,\B}H\sum_{\gri\in\inc_N^m}\abs{f_{c,\gr{k};\gri}
\pr{ \pr{x_{n;\grj}}_{\grj
\in\inc^m_n }
}}^q
\\
+K\pr{m,p,q,\B}H
\sum_{\emptyset \subsetneq J\subsetneq \intent{1,c}}\sum_{\gr{i_J}\in\N^J
}
\E{\pr{ \sum_{\gr{i_{J^c}}: \gr{i_J}+\gr{i_{J^c}}\in
      \inc_N^c}\abs{f_{c,\gr{k};\gr{i_J}+\gr{i_{J^c}}}
\pr{ \pr{x_{n;\grj}}_{\grj
    \in\inc^m_n }    }}^p   }^{q/p}}\\
+K\pr{m,p,q,\B}H \pr{\sum_{\gri\in\inc_n^c }\abs{f_{c,\gr{k};\gri}
\pr{ \pr{x_{n;\grj}}_{\grj
    \in\inc^m_n }    }}^p
  }^{q/p}.
\end{multline}
Integrating over the $C^m_n$ variables
$x_{n;\grj}$,
$\grj\in\inc^m_n$ with respect to
$\mu_0$, we get that
\begin{multline}\label{eq:etape_control_moments_Ustatistique_incomp}
 \E{  \norm{\sum_{\gri\in\inc^c_n    }
a_{n,c,\gr{k};\gri}h^{\pr{c}}\pr{\xi_{\gri} }
}_{\B}^q  } 
\leq
K\pr{m,p,q,\B}H\sum_{\gri\in\inc_n^m}\E{\abs{a_{n,c,\gr{k};\gri}}^q}
\\
+K\pr{m,p,q,\B}H
\sum_{\emptyset \subsetneq J\subsetneq \intent{1,c}}\sum_{\gr{i_J}\in\N^J
}
\E{\pr{ \sum_{\gr{i_{J^c}}: \gr{i_J}+\gr{i_{J^c}}\in
      \inc_n^c}\abs{ a_{n,c,\gr{k};\gr{i_J}+\gr{i_J^c}}}^p   }^{q/p}}\\
+K\pr{m,p,q,\B}H \E{\pr{\sum_{\gri\in\inc_n^c 
}\abs{a_{n,c,\gr{k};\gri}}^p
}^{q/p}}.
\end{multline}
We are thus reduced to bound the moments of order $q/p$ of
random variables of the form
\begin{equation}\label{eq:va_somme_de_bernoulli}
  Y=\sum_{a\in A}\pr{\sum_{b\in B}Y_{a,b} }^p,
\end{equation}
where $A$ and $B$ are finite sets of respective cardinal $\abs{A}$
and $\abs{B}$, and $\pr{Y_{a,b}}_{a\in A,b\in B}$ is independent.
\begin{Lemma}\label{lem:moments_of_Y}
  Let $Y$ be a random variable of the form
  \eqref{eq:va_somme_de_bernoulli}, where $Y_{a,b}$ has a Bernoulli
  distribution of parameter $y\in [0,1]$, that is,
  $\PP\pr{Y_{a,b}=1}=y$ and $\PP\pr{Y_{a,b}=0}=1-y$.
  There exists a constant $C_{p,q}$ such that for $q\geq p$,
  \begin{equation}
    \E{Y^{q/p}}\leq
C_{p,q}\pr{\abs{A}^{q/p}\abs{B}^qy^q+\abs{A}^{q/p}
  \abs{B}^{q/p}y^{q/p}+\abs{A}\abs{B}y }
\end{equation}
\end{Lemma}
\begin{proof}
We will use the following moment inequality for
sums of independent non-negative random variables, given in
Corollary~3 in \cite{MR1457628}: for each $s\geq 1$,
there exists a constant $K_s$ such that for each
finite set of independent non-negative random variables
$\pr{X_i}_{i\in I}$,
\begin{equation}\label{eq:moment_inequality_non_negative}
  \E{\pr{\sum_{i\in I}X_i }^s }
  \leq K_s \pr{\sum_{i\in I}\E{X_i}}^s+
  K_s \sum_{i\in I}\E{X_i^s}.
\end{equation}
Applying \eqref{eq:moment_inequality_non_negative}
with $s=q/p$, $I=A$ and $X_i=\pr{\sum_{b\in B}Y_{i,b}}^p$ gives
\begin{equation}\label{eq:premiere_borne_Y}
  \E{Y^{q/p}}\leq K_{q/p}
  \pr{\sum_{a\in A}\E{\pr{\sum_{b\in B}Y_{a,b}}^p}
  }^{q/p}+K_{q/p}\sum_{a\in A}\E{\pr{\sum_{b\in B}Y_{a,b} }^q }.
\end{equation}
Applying \eqref{eq:moment_inequality_non_negative}
with $s=p$, $I=B$ and $X_i=Y_{a,i}$ gives
\begin{equation}\label{eq:borne_somme_puissance_s_des_esperances}
 \E{\pr{\sum_{b\in B}Y_{a,b}}^p}
 \leq K_p \pr{\sum_{b\in B}\E{Y_{a,b}}   }^p
 +K_p\sum_{b\in B}\E{Y_{a,b}^p}
 =K_p \pr{\abs{B}^p y^p +\abs{B}y}.
\end{equation}
Applying \eqref{eq:moment_inequality_non_negative}
with $s=q$, $I=B$ and $X_i=Y_{a,i}$ gives
\begin{equation}\label{eq:bornes_esperance_des_puissances}
 \E{\pr{\sum_{b\in B}Y_{a,b} }^q }
 \leq K_q\pr{\sum_{b\in B}\E{Y_{a,b}}}^q
 +K_q\sum_{b\in B}\E{Y_{a,b}^q}=
 K_q\pr{\abs{B}^py^q+\abs{B}y}.
\end{equation}
The combination of \eqref{eq:premiere_borne_Y},
\eqref{eq:borne_somme_puissance_s_des_esperances},
\eqref{eq:bornes_esperance_des_puissances} and 
 inequality $\abs{A}^{q/p}\geq\abs{A}$ ends the
  proof of Lemma~\ref{lem:moments_of_Y}.
\end{proof}
We use Lemma~\ref{lem:moments_of_Y} in order to bound each term of
\eqref{eq:etape_control_moments_Ustatistique_incomp}.
Notice that the random variables $a_{n,c,\gr{k};\gri}$ have
the same distribution. Moreover, each of these random variables is a
sum of a number that does not exceed $n^{m-c}$
hence by Lemma~\ref{lem:moments_of_Y} with sets $A$ and $B$ of
respective cardinal one and $n^{m-c}$ gives
\begin{equation}
\label{eq:borne_U_stat_inc_somme_moments_ordre_q}
 \sum_{\gri\in\inc_n^m}\E{\abs{a_{n,c,\gr{k};\gri}}^q}
\leq C_{p,q}n^c\pr{ n^{\pr{m-c}q}p_n^q
  +n^{\pr{m-c}q/p}p_n^{q/p} +n^{m-c}p_n}
\end{equation}
Since $q/p\geq 1$ and $q\geq 1$, the previous bound is
non-increasing in $c$ and reaches its maximum for $c=d$ hence the
bound \eqref{eq:borne_U_stat_inc_somme_moments_ordre_q} can be
converted into a bound independent of $c$, namely,
\begin{equation}\label{eq:ustats_inc_borne1}
  \sum_{\gri\in\inc_n^m}\E{\abs{a_{n,c,\gr{k};\gri}}^q}\leq 
C_{p,q}\pr{n^{\pr{m-d}q+d }p_n^q+
    n^{\pr{m-d}q/p+d}p_n^{q/p}+n^mp_n}.
\end{equation}
For the second term of the right hand side of
\eqref{eq:etape_control_moments_Ustatistique_incomp},
consider $\emptyset \subsetneq J\subsetneq \intent{1,c}$
of cardinal $j\in\intent{1,c-1}$. The sum over
$\gr{i_J}$ consists of at most $n^j$ terms
and the involved expectation can be bounded
via Lemma~\ref{lem:moments_of_Y} where
$A$ has at most $n^{c-k}$ elements
and $B$ has at most $n^{m-c}$ elements. We thus obtain
\begin{multline}
\sum_{\emptyset \subsetneq J\subsetneq \intent{1,c}}\sum_{\gr{i_J}\in\N^J
}
\E{\pr{ \sum_{\gr{i_{J^c}}: \gr{i_J}+\gr{i_{J^c}}\in
      \inc_N^c}\abs{ a_{n,c,\gr{k};\gr{i_J}+\gr{i_J^c}}}^p   }^{q/p}}\\
\leq C_{p,q}\sum_{j=1}^{c-1}
n^j \pr{ n^{\pr{c-j}q/p}n^{\pr{m-c}q}p_n^q+
  n^{\pr{c-j}q/p}n^{\pr{m-c}q/p}p_n^{q/p}
  +n^{c-j}n^{m-c}p_n }\\
=C_{p,q} \pr{n^{mq}p_n^{q}n^{cq\pr{1/p-1}}\sum_{j=1}^{c-1}n^{j\pr{1-q/p}}
+ n^{mq/p}p_n^{q/p}\sum_{j=1}^{c-1}n^{j\pr{1-q/p}}
+ n^{m}p_n}.
\end{multline}
Since $q/p\geq 1$, monotonicity of the previous quantities in $k$ and
$c$ leads to the bound
\begin{multline}\label{eq:ustats_inc_borne2}
 \sum_{\emptyset \subsetneq J\subsetneq \intent{1,c}}
 \sum_{\gr{i_J}\in\N^J}
\E{\pr{ \sum_{\gr{i_{J^c}}: \gr{i_J}+\gr{i_{J^c}}\in
      \inc_n^c}\abs{ a_{n,c,\gr{k};\gr{i_J}+\gr{i_J^c}}}^p   }^{q/p}}\\
\leq
C_{p,q}\pr{n^{mq+q/p-1}p_n^q+ n^{\pr{m-1}q/p+1}p_n^
{ q / p } +C_{p, q } n ^ { m }p_n}\\
\leq C_{p,q}\pr{n^{mq+d\pr{q/p-1}}p_n^q+  n^{mq/p}
p_n^{ q / p } +n^{m }p_n}.
\end{multline}
Finally, the last term of the right hand side of
\eqref{eq:etape_control_moments_Ustatistique_incomp}
is controlled via an application of Lemma~\ref{lem:moments_of_Y} with
a set $A$ of cardinal smaller than $n^c$
and a set $B$ of cardinal smaller than $n^{m-c}$, which gives
\begin{equation}
 \E{\pr{\sum_{\gri\in\inc_n^c }\abs{a_{n,c,\gr{k};\gri}}^p
}^{q/p}}
\leq C_{p,q}\pr{n^{cq/p}n^{\pr{m-c}q}p_n^q
+ n^{cq/p}n^{\pr{m-c}q/p}p_n^{q/p}+
n^mp_n},
\end{equation}
and since the last quantity is maximal at $c=d$,
\begin{equation}\label{eq:ustats_inc_borne3}
 \E{\pr{\sum_{\gri\in\inc_N^c }\abs{a_{n,c,\gr{k};\gri}}^p
}^{q/p}}
\\
\leq C_{p,q}\pr{n^{dq/p}n^{\pr{m-d}q}p_n^q
+  n^{m q/p}p_n^{q/p}+
n^mp_n}.
\end{equation}
We conclude the proof of
Theorem~\ref{thm:moment_inequality_incomplete_Ustat}
by collecting the bounds \eqref{eq:ustats_inc_borne1},
\eqref{eq:ustats_inc_borne2} and \eqref{eq:ustats_inc_borne3}.
  \begin{appendix}
 \section{ Elementary properties of conditional expectation}
 
We collect some lemmas which will be used in the proofs.

The following fact on conditional expectation is well-known.
\begin{Lemma}\label{lem:removing_ind_cond_exp}
Let $Y$ be an integrable random variable and let $\Aca$ and $\Bca$ be two $\sigma$-algebras such that $\Bca$ is independent of $\sigma\pr{Y}\vee \Aca$. 
Then 
\begin{equation}
\E{Y\mid\Aca\vee\Bca}=\E{Y\mid\Aca}.
\end{equation}
\end{Lemma}
\begin{Lemma}\label{lem:two_cond_exp}
Let $\pr{\xi_i}_{i\in\Z}$ be an independent sequence and let $\pr{\xi'_i}_{i\in\Z}$ be an independent copy of $\pr{\xi_i}_{i\in\Z}$. Denote for 
$I, J\subset \Z$ by $\Fca_I$ (respectively $\Fca'_J$) the $\sigma$-algebra generated by the random variables 
$\xi_i,i\in I$ (respectively $\xi'_j,j\in J$). For each integrable random variable $Y$ and each subsets $I_1,I_2$ and $J_1,J_2$ of $\Z$, 
\begin{equation}
\E{\E{Y\mid\Fca_{I_1}\vee \Fca'_{J_1}}\mid \Fca_{I_2}\vee \Fca'_{J_2}}=\E{Y\mid \Fca_{I_1\cap I_2} \vee \Fca'_{J_1\cap J_2}}.
\end{equation}
\end{Lemma} 
 \begin{proof}
 Expressing $\Fca'_{J_2}$ as $\Fca'_{J_1\cap J_2}\vee\Fca'_{J_2\setminus I_1} $ and 
 applying Lemma~\ref{lem:removing_ind_cond_exp} with 
 $\widetilde{Y}=\E{Y\mid\Fca_{I_1}\vee \Fca'_{J_1}}$, $\Aca=\Fca_{I_2}\vee \Fca'_{J_1\cap J_2}$ 
 and $\Bca=\Fca'_{J_2\setminus J_1}$, we get that 
 \begin{equation}
\E{\E{Y\mid\Fca_{I_1}\vee \Fca'_{J_1}}\mid \Fca_{I_2}\vee \Fca'_{J_2}}=
\E{\E{Y\mid\Fca_{I_1}\vee \Fca'_{J_1}}\mid \Fca_{I_2}\vee \Fca'_{J_1\cap J_2}}.
\end{equation}
Then writing $ \Fca_{I_2}$ as $\Fca_{I_1\cup I_2}\vee \Fca_{I_2\setminus I_1}$
and applying Lemma~\ref{lem:removing_ind_cond_exp} with 
 $\widetilde{Y}=\E{Y\mid\Fca_{I_1}\vee \Fca'_{J_1}}$, $\Aca=\Fca_{I_1\cap I_2}\vee \Fca'_{J_1\cap J_2}$ 
 and $\Bca=\Fca_{I_2\setminus I_1}$ gives 
  \begin{equation}
\E{\E{Y\mid\Fca_{I_1}\vee \Fca'_{J_1}}\mid \Fca_{I_2}\vee \Fca'_{J_2}}=
\E{\E{Y\mid\Fca_{I_1}\vee \Fca'_{J_1}}\mid \Fca_{I_1\cap I_2}\vee \Fca'_{J_1\cap J_2}}
\end{equation}
and the inclusion $ \Fca_{I_1\cap I_2}\vee \Fca'_{J_1\cap J_2}\subset 
\Fca_{I_1}\vee \Fca'_{J_1}$ allows to conclude. 
 \end{proof}
 
\end{appendix}
\def\polhk\#1{\setbox0=\hbox{\#1}{{\o}oalign{\hidewidth
  \lower1.5ex\hbox{`}\hidewidth\crcr\unhbox0}}}\def\cprime{$'$}
  \def\polhk#1{\setbox0=\hbox{#1}{\ooalign{\hidewidth
  \lower1.5ex\hbox{`}\hidewidth\crcr\unhbox0}}} \def\cprime{$'$}

\end{document}